\documentclass[psamsfonts]{amsart}

\usepackage{amssymb,amsfonts,amscd}
\usepackage[all,arc]{xy}
\usepackage{enumerate}
\usepackage{mathrsfs}

\newtheorem{thm}{Theorem}[section]
\newtheorem{cor}[thm]{Corollary}
\newtheorem{prop}[thm]{Proposition}
\newtheorem{lem}[thm]{Lemma}
\newtheorem{conj}[thm]{Conjecture}

\newtheorem*{thm1}{Theorem A}
\newtheorem*{thm2}{Theorem B}

\theoremstyle{definition}
\newtheorem{defn}[thm]{Definition}

\theoremstyle{remark}
\newtheorem{rem}[thm]{Remark}

\makeatletter
\let\c@equation\c@thm
\makeatother
\numberwithin{equation}{section}

\newcommand{\upperRomannumeral}[1]{\uppercase\expandafter{\romannumeral#1}}

\title[]{Nonnegative Ricci curvature, almost stability at infinity, and structure of fundamental groups}

\author[]{Jiayin Pan}

\newcommand{\Addresses}{{
		\bigskip
		\footnotesize
		
		Jiayin Pan, \textsc{Department of Mathematics, University of California\ -\ Santa Barbara, CA, USA.}\par\nopagebreak
		\textit{E-mail address}: \texttt{jypan10@gmail.com}

}}

\begin{document}
\begin{abstract}
	We study the fundamental group of an open $n$-manifold $M$ of nonnegative Ricci curvature with additional stability condition on $\widetilde{M}$, the Riemannian universal cover of $M$. We prove that if any tangent cone of $\widetilde{M}$ at infinity is a metric cone, whose cross-section is sufficiently Gromov-Hausdorff close to a prior fixed metric space, then $\pi_1(M)$ is finitely generated and contains a normal abelian subgroup of finite index; if in addition $\widetilde{M}$ has Euclidean volume growth of constant at least $L$, then we can bound the index of that abelian subgroup in terms of $n$ and $L$. In particular, our result implies that if $\widetilde{M}$ has Euclidean volume growth of constant at least $1-\epsilon(n)$, then $\pi_1(M)$ is finitely generated and $C(n)$-abelian.
\end{abstract}
	
\maketitle

In Riemannian geometry, one of the most important problems is to study the interplay between curvature and topology. For example, a classical Bieberbach theorem states that the fundamental group of any complete flat $n$-manifold contains a normal free abelian subgroup of index at most $C(n)$, a constant only depends on $n$. For open $n$-manifolds with non-negative Ricci curvature, a longstanding problem is the Milnor conjecture raised in 1968: any open $n$-manifold of nonnegative Ricci curvature has a finitely generated fundamental group \cite{Mi}. This conjecture remains open today. The Milnor conjecture has been partially confirmed under various conditions; see \cite{An,Li,Liu,Pan1,Pan2,Sor,Wu}.

To understand the fundamental group of a manifold $M$, it is natural to consider the Riemannian universal cover of $M$, where the fundamental group acts as isometries. For instance, since any flat $n$-manifold has Riemannian universal cover $\mathbb{R}^n$, the standard Euclidean space, to study the fundamental group of a flat manifold, it is equivalent to study isometric actions on $\mathbb{R}^n$. As one of the consequences of the main result of this paper, we describe the fundamental group of an open manifold of nonnegative Ricci curvature, whose Riemannian universal cover is close to $\mathbb{R}^n$ in terms of volume growth.

\begin{thm1}
	Given $n\in\mathbb{N}$, there are positive constants $\epsilon(n)$ and $C(n)$ such that the following holds.
	
	Let $(M,x)$ be an open $n$-manifold of $\mathrm{Ric}\ge 0$. If $(\widetilde{M},\tilde{x})$, the Riemannian universal cover of $(M,x)$, has Euclidean volume growth of constant at least $1-\epsilon(n)$, that is,
	$$\limsup\limits_{r\to\infty}\dfrac{\mathrm{vol}(B_r({\tilde{x}}))}{\mathrm{vol}(B_r^n(0))}\ge 1-\epsilon(n),$$
	then $\pi_1(M)$ is finitely generated and contains a normal free abelian subgroup of index at most $C(n)$.
\end{thm1}

In fact, we will prove that a similar conclusion holds for a class of manifolds under a much weaker condition in terms of the geometric stability of $\widetilde{M}$ at infinity (see Theorem B).

The first part of Theorem A partially confirms the Minor conjecture. It is worth mentioning that in Theorem A, the number of generators can be uniformly bounded by some constant $C(n)$; this follows from finite generation and the work of Kapovitch and Wilking \cite{KW}. Our proof uses the strategy introduced in \cite{Pan2}: studying the stability of fundamental group action on $\widetilde{M}$ at infinity via equivariant Gromov-Hausdorff convergence, which we will explain in details later.
In \cite{Pan2}, we studied an open manifold of nonnegative Ricci curvature, whose universal cover $\widetilde{M}$ is $k$-Euclidean at infinity, and showed that the fundamental group is finitely generated. The $k$-Euclidean at infinity condition means that any tangent cone of $\widetilde{M}$ at infinity is a metric cone splitting of a maximal Euclidean factor $\mathbb{R}^k$. However, this cannot be applied to Theorem A, where tangent cones of $\widetilde{M}$ may not have the same maximal Euclidean factor.

The second part of Theorem A generalizes the classical Bieberbach theorem. By \cite{Mi,Gro2}, if $M$ has nonnegative Ricci curvature, then any finitely generated subgroup of $\pi_1(M)$ contains a nilpotent subgroup of finite index (also see \cite{KW}). Recall that for any open manifold of non-negative sectional curvature, Cheeger and Gromoll demonstrated that the fundamental group is always virtually abelian \cite{CG}; later we will see that our main result also generalizes this one (see Theorem B(2,3)). However, an open manifold of non-negative Ricci curvature may admit a torsion free nilpotent non-abelian group, which does not contain any abelian subgroup of finite index, as its fundamental group \cite{Wei}. Thus one can only expect the virtually abelian structure to be true with additional conditions (see Corollaries \ref{cor_v_abel} and \ref{cor_C_abel}).

To state our main result in full generality, we introduce an almost stability condition below. We denote $\mathcal{M}(n,0)$ as the set of all Gromov-Hausdorff limit spaces coming from a sequence of complete $n$-manifolds $(M_i,p_i)$ of $\mathrm{Ric}\ge 0$; we denote $d_{GH}$ as the Gromov-Hausdorff distance.
\begin{defn}\label{def_al_stable}
	Let $\epsilon>0$ and let $C(X)\in \mathcal{M}(n,0)$ be a metric cone. For an open $n$-manifold $M$ of $\mathrm{Ric}\ge 0$, we say that $M$ is \textit{$(C(X),\epsilon)$-stable at infinity}, if any tangent cone of $M$ at infinity is a metric cone $(C(Z),z)$ with vertex $z$ and $d_{GH}(Z,X)\le\epsilon$. 
\end{defn}

Roughly speaking, this almost stability condition requires that all tangent cones of $M$ at infinity are alike metric cones. If $\widetilde{M}$ satisfies the volume growth condition in Theorem A, then by \cite{CC1} any tangent cone of $\widetilde{M}$ at infinity is a metric cone $(C(Z),z)$ with vertex $z$ and $d_{GH}(Z,S^{n-1})\le \epsilon'$ for some $\epsilon'>0$, where $S^{n-1}$ is the $n-1$ dimensional unit sphere; in other words, $\widetilde{M}$ is $(C(S^{n-1}),\epsilon')$-stable at infinity. If $M$ has nonnegative sectional curvature, or $\widetilde{M}$ has Euclidean volume growth with a unique tangent cone at infinity, then $\widetilde{M}$ also satisfies Definition \ref{def_al_stable} with $\epsilon=0$ since any $\widetilde{M}$ has a unique tangent cone as a metric cone at infinity \cite{BGP,CC1}. Also note that Definition \ref{def_al_stable} does not require that all tangent cones at infinity have the same dimension. For example, metric cones $C(S^2_1)$ and $C(S^2_{1}\times S^2_\epsilon)$ have Gromov-Hausdorff close cross-sections when $\epsilon$ is small, where $S_r^2$ is the round $2$-sphere of radius $r$. Even when all tangent cones of $\widetilde{M}$ at infinity have dimension $n$, the almost stability condition allows these tangent cones to have non-homeomorphic cross-sections (see examples in \cite{CN2}). We also compare Definition \ref{def_al_stable} with $k$-Euclidean at infinity condition in \cite{Pan2}, which describes the stability of Euclidean factor at infinity. While the $k$-Euclidean condition and the almost stability condition share some overlap (for example, unique tangent cone at infinity as a metric cone), only the later one can be applied to manifolds in Theorem A, where $\widetilde{M}$ may have different dimensions of maximal Euclidean factors at infinity.

We state our main result.

\begin{thm2}
	Given a metric cone $C(X)\in\mathcal{M}(n,0)$, there is a constant $\epsilon_X>0$ such that the following holds.
	
	Let $M$ be an open $n$-manifold of $\mathrm{Ric}\ge 0$ and let $\widetilde{M}$ be the Riemannian universal cover of $M$. If $\widetilde{M}$ is $(C(X),\epsilon_X)$-stable at infinity, then the followings hold.\\
	(1) $\pi_1(M)$ is finitely generated.\\
	(2) $\pi_1(M)$ contains a normal abelian subgroup of finite index.\\
	(3) If in addition $\widetilde{M}$ has Euclidean volume growth of constant at least $L$, then the index in (2) can be bounded by $C(n,L)$, a constant only depending on $n$ and $L$.
\end{thm2}

We will prove that Theorem B(2,3) holds for the $k$-Euclidean condition as well (see Corollaries \ref{cor_v_abel} and \ref{cor_C_abel}). 

As indicated, Theorem A is a special case of Theorem B.

\begin{proof}[Proof of Theorem A by Theorem B]
	Let $S$ be the unit $(n-1)$-dimensional sphere and $\epsilon_{S}>0$ be the corresponding constant in Theorem B. The metric cone $C(S)$ is the standard $n$-dimensional Euclidean space $\mathbb{R}^n$. By \cite{CC1}, we can choose a constant $\epsilon(n)>0$ such that for any complete manifold $(N,q)$ with $\mathrm{Ric}_N \ge 0$, if
	$${\mathrm{vol}(B_1(q))}\ge (1-\epsilon(n)){\mathrm{vol}(B_1^n(0))},$$
	then
	$$d_{GH}(B_1(q),B_1^n(0))\le \epsilon_S.$$
	With this constant $\epsilon(n)$, relative volume comparison, and the volume growth condition on $\widetilde{M}$, we see that
	$$d_{GH}(r^{-1}B_r(\tilde{x}),B_1^n(0))\le \epsilon_S.$$
	for all $r>0$. For any sequence $r_i\to\infty$ and the convergence
	$$(r_i^{-1}\widetilde{M},\tilde{x})\overset{GH}\longrightarrow (Y,y),$$
	the limit space $(Y,y)$ is a metric cone $(C(Z),z)$ with vertex $z$ \cite{CC1}. We also obtain that $d_{GH}(Z,S)\le\epsilon_S$ from the convergence. Applying Theorem B, we conclude that $\pi_1(M)$ is finite generated and contains a normal abelian subgroup of index at most $C(n)$. To see that this abelian subgroup is torsion free, we further shrink $\epsilon(n)$ if necessary so that the volume growth condition guarantees that $\widetilde{M}$ is diffeomorphic to $\mathbb{R}^n$ \cite{CC1}. Since $\mathbb{R}^n$ does not admit a nontrivial finite group action as covering transformations \cite{Bre}, the result follows.
\end{proof}

Based on Theorems A and B, we propose the following conjecture.

\begin{conj}\label{my_conj}
	Given $n\in\mathbb{N}$ and $L\in(0,1]$, there exists a constant $C(n,L)$ such that the following holds.
	
	Let $M$ be an open $n$-manifold of $\mathrm{Ric}\ge 0$. If the Riemannian universal cover of $M$ has Euclidean volume growth of constant at least $L$, then $\pi_1(M)$ is finitely generated and contains a normal abelian subgroup of index at most $C(n,L)$.
\end{conj}

Establishing the equivariant stability at infinity is the key to prove Theorem B. More precisely, for any sequence $r_i\to\infty$, passing to a subsequence if necessary, we study the equivariant Gromov-Hausdorff convergence \cite{FY}:
$$(r_i^{-1}\widetilde{M},\tilde{x},\pi_1(M,x))\overset{GH}\longrightarrow(\widetilde{Y},\tilde{y},G),$$
where $G$ acts effectively on $\widetilde{Y}$ by isometries.
We call $(\widetilde{Y},\tilde{y},G)$ an equivariant tangent cone of $(\widetilde{M},\pi_1(M,x))$ at infinity. In general, $(\widetilde{Y},\tilde{y},G)$ depends on the scaling sequence $\{r_i\}$. To study the set of all equivariant tangent cones at infinity, we use the structure results of Ricci limit spaces developed by Cheeger, Colding, and Naber \cite{CC1,CC2,CC3,CN}, especially the result that the isometry group of any Ricci limit space is a Lie group \cite{CC2,CN}. The key is to show that under such a geometric stability condition at infinity, $G$-action has certain equivariant stability at infinity as well. In \cite{Pan2}, we showed that if $\widetilde{M}$ is $k$-Euclidean at infinity and $\pi_1(M)$ is abelian, then the projection of $G$-action on the maximal Euclidean factor in $Y$ is independent of the sequence $\{r_i\}$.

\begin{thm}\label{omega_k_eu}\cite{Pan2}
	Let $M$ be an open $n$-manifold of $\mathrm{Ric}\ge 0$. Suppose that $\pi_1(M)$ is abelian and the universal cover $\widetilde{M}$ is $k$-Euclidean at infinity. Then there exist a closed abelian subgroup $K$ of $O(k)$ and an integer $l\in[0,k]$ such that the following holds. 
	
	Let $$(\widetilde{Y},\tilde{y},G)=(\mathbb{R}^k\times C(Z),(0,z),G)$$ be an equivariant tangent cone of $(\widetilde{M},\pi_1(M,x))$ at infinity, where $\mathrm{diam}(Z)<\pi$. Then the projected $G$-action on $\mathbb{R}^k$-factor $(\mathbb{R}^k,0,p(G))$ satisfies $p(G)=K\times \mathbb{R}^l$ with $K$ fixing $0$ and the subgroup $\{e\}\times \mathbb{R}^l$ acting as translations in the $\mathbb{R}^k$-factor, where $p:\mathrm{Isom}(\mathbb{R}^k\times C(Z))\to \mathrm{Isom}(C(Z))$ is the natural projection.
\end{thm} 

In particular, it follows from Theorem \ref{omega_k_eu} that the limit orbit $G\cdot \tilde{y}$ is always an $l$-dimensional Euclidean subspace regardless of the scaling sequence $\{r_i\}$. By Lemma 2.5 in \cite{Pan1} and Wilking's reduction \cite{Wi}, the connectivity of the orbit $G\cdot \tilde{y}$ confirms the Milnor conjecture when $\widetilde{M}$ is $k$-Euclidean at infinity.

When $\widetilde{M}$ is $(C(X),\epsilon_X)$-stable at infinity, since different tangent cones of $\widetilde{M}$ at infinity may have different maximal Euclidean factors, one can not expect that they have the same projection of $G$-action to the maximal Euclidean factor. Nonetheless, we show that the limit orbit at the base point is independent of the scaling sequence: it is always an $l$-dimensional Euclidean subspace.

\begin{thm}\label{main_omega}
	Let $C(X)\in \mathcal{M}(n,0)$ be a metric cone and let $\epsilon_X>0$ be the constant in Theorem B. For any open $n$-manifold $M$ with $\mathrm{Ric}\ge 0$, suppose that its universal cover $\widetilde{M}$ is $(C(X),\epsilon_X)$-stable at infinity and $\pi_1(M)$ is nilpotent, then there exists an integer $l\in[0,n]$ such that any equivariant tangent cone of $(\widetilde{M},\pi_1(M,x))$ at infinity $(C(Z),z,G)$ has orbit $G\cdot z$ as an $l$-dimensional Euclidean subspace of $C(Z)$.
\end{thm}

As pointed out, Theorem \ref{main_omega} implies Theorem B(1). A more detailed description of $G$-action on $C(Z)$ will be given in Section 3 as Theorem \ref{main_omega'}. In fact, in order to prove Theorem \ref{main_omega}, it is essential to know about the $G$-action other than its orbit at $z$, for example, the isotropy subgroup at $z$. To understand the stability among these isotropy subgroups in different limits, we need to look for certain stability among group actions on a family of alike metric cones with possibly different Euclidean factors. 

In Theorem \ref{main_omega}, we assume that $\pi_1(M)$ is nilpotent. Though it is sufficient to consider abelian fundamental groups to prove finite generation with the help of Wilking's reduction \cite{Wi}, the nilpotent situation will be applied to prove Theorem B(2,3). By Theorem B(1) and \cite{KW}, $\pi_1(M)$ contains a normal nilpotent subgroup of index at most $C(n)$ under the assumptions of Theorem B. Hence in order to prove Theorem B(2,3), we are free to assume that $\pi_1(M)$ itself is nilpotent. The equivariant stability at infinity restricts the behavior of any element in $\pi_1(M,x)$ with sufficiently large displacement at $\tilde{x}$. Indeed, we show that such an element $\gamma$ behaves almost as a translation at $\tilde{x}$, in the sense that $d(\gamma^2\tilde{x},\tilde{x})$ is close to $2d(\gamma\tilde{x},\tilde{x})$ (see Lemma \ref{almost_trans}). This is the key geometric input to prove Theorem B(2).

We indicate our proof of Theorem \ref{main_omega}. To understand the equivariant stability at infinity, we study $\Omega(\widetilde{M},\Gamma)$, the set of all equivariant tangent cones of $(\widetilde{M},\Gamma)$ at infinity, which is a compact and connected set with respect to the equivariant Gromov-Hausdorff topology. Two technical tools are developed in \cite{Pan2} to achieve this goal. The first one is an equivariant Gromov-Hausdorff distance gap between different isometric group actions on a fixed closed Riemannian manifold (see Proposition 3.1 in \cite{Pan2}). The second one is a critical rescaling argument (see Sections 2 and 4 in \cite{Pan2}). To make use of these tools to prove Theorem \ref{main_omega}, we further improve them based on some new ideas.

We explain why new improvements are crucial. For the $k$-Euclidean case as Theorem 0.4, it is sufficient to apply the distance gap to isometric actions on the unit sphere $S^{k-1}_1$, which is the cross-section of $\mathbb{R}^k$. Essentially, we deal with group actions on a fixed metric cone $\mathbb{R}^k=C(S^{k-1}_1)$ after some reductions. However, for the $(C(X),\epsilon)$-stable case, we have to deal group actions on a family of metric cones with alike cross-sections. We develop an equivariant Gromov-Hausdorff distance gap among isometric actions on alike cross-sections. For a model cross-section $X$ and a cross-section $Z$ with $d_{GH}(Z,X)\le\epsilon$, we show that when $\epsilon$ is sufficient small, any isometric $H$-action on $Z$ naturally corresponds to a unique isometric $G$-action on $X$ (see Proposition \ref{to_model}). We call $(X,G)$ as the \textit{$X$-mark} of $(Z,H)$. For example, considering $X=S^2_1$ as the model space and $Z=S^2_{1}\times S^2_{\epsilon}$ with small $\epsilon$, for a $S^1\times S^1$ rotational action on $Z$, we can naturally correspond it to a $S^1$ rotational action on the model $X$; thus $(Z,S^1\times S^1)$ has $S^2_1$-mark $(S^2_1,S^1)$. We use this $X$-mark to compare group actions on spaces being close to $X$ and establish the equivariant Gromov-Hausdorff distance gap (see Theorem \ref{isom_stable}). Note that the $X$-mark may not distinguish different actions on $Z$: for example $(S^2_{1}\times S^2_{\epsilon},S^1\times \{e\})$ and $(S^2_{1}\times S^2_{\epsilon},S^1\times S^1)$ have the same $S^2_1$-mark. To handle this ambiguity when proving Theorem \ref{main_omega}, we also need to largely modify the critical rescaling argument (see Section 3 for more details).

We organize the paper as below. In Section 1, after briefly recalling some basic facts on equivariant Gromov-Hausdorff convergence, we develop a method to study isometric actions on spaces that are very close to a fixed space; we also prove the equivariant Gromov-Hausdorff gap mentioned before (see Theorem \ref{isom_stable}). We study certain isometric actions on metric cones with special properties in Section 2, as preparations to prove a detailed version of Theorem \ref{main_omega} by using critical rescaling arguments in Section 3 (see Theorem \ref{main_omega'}). We prove Theorem B(2) and (3) in Sections 4 and 5 respectively, with the $k$-Euclidean case included (see Corollaries \ref{cor_v_abel} and \ref{cor_C_abel}).

\tableofcontents

The author would like to thank professor Xiaochun Rong for his encouragement and numerous helpful discussions. The author would like to thank professor Guofang Wei for her suggestions when preparing the paper.\\

We use the notations below throughout the paper:\\
$\cdot$ $(M,x)$ a pointed complete Riemannian manifold.\\
$\cdot$ $(\widetilde{M},\tilde{x})$ the Riemannian universal cover of $(M,x)$.\\
$\cdot$ $\pi_1(M,x)$ the fundamental group of $M$ at $x$.\\
$\cdot$ $\mathrm{Isom}(X)$ the isometry group of a metric space $X$.\\
$\cdot$ $(X,x,G)$ a pointed metric space $(X,x)$ with a closed subgroup $G\subseteq \mathrm{Isom}(X)$. In particular, this means that $G$-action is effectively on $X$.\\
$\cdot$ $G\cdot x$ the $G$-orbit at $x$.\\
$\cdot$ $\mathrm{Iso}_xG$ the isotropy subgroup of $G$ at $x$.\\
$\cdot$ $C(Z)$ the metric cone over a compact metric space $Z$.\\
$\cdot$ $\mathcal{M}(n,\kappa)$ the set of all Ricci limit spaces coming from a sequence of complete $n$-manifolds $(M_i,p_i)$ of $\mathrm{Ric}\ge (n-1)\kappa$.\\
$\cdot$ $\mathcal{M}(n,\kappa,v)$ the set of all Ricci limit spaces coming from a sequence of complete $n$-manifolds $(M_i,p_i)$ of $\mathrm{Ric}\ge (n-1)\kappa$ and $\mathrm{vol}(B_1(p_i))\ge v$.\\
$\cdot$ $\mathcal{M}_{cs}(n,0)=\{Z\ |\ C(Z)\in\mathcal{M}(n,0)\}$.\\
$\cdot$ $\mathcal{M}_{cs}(n,0,v)=\{Z\ |\ C(Z)\in\mathcal{M}(n,0,v)\}$.\\
$\cdot$ $\overset{H}\longrightarrow$ Hausdorff convergence.\\
$\cdot$ $\overset{GH}\longrightarrow$ Gromov-Hausdorff convergence or equivariant Gromov-Hausdorff convergence, depending on the context.\\
$\cdot$ $d_{GH}$ Gromov-Hausdorff distance or equivariant Gromov-Hausdorff distance, depending on the context.\\
$\cdot$ $\Omega(M,G)$ the set of all equivariant tangent cones of $(M,x,G)$ at infinity, endowed with equivariant Gromov-Hausdorff topology.\\
$\cdot$ $[\alpha,\beta]=\alpha\beta\alpha^{-1}\beta^{-1}$ the commutator of two elements $\alpha$ and $\beta$ in a group.\\ 
$\cdot$ $C_1(G)=[G,G]$ the subgroup of $G$ generated by all commutators.\\
$\cdot$ $C_{k+1}(G)=[C_k(G),G]$, $k\ge 1$.\\
$\cdot$ $Z(G)$ the center of $G$.\\
$\cdot$ $[G:H]$ the index of subgroup $H$ in $G$.

\section{Stability of isometric actions among alike cross-sections}
	
In this section, we fix a cross-section $X\in\mathcal{M}_{cs}(n,0)$ as a model space. Note that by splitting theorem \cite{CC1}, $\mathrm{diam}(X)\le \pi$ for all $X\in\mathcal{M}_{cs}(n,0)$. The main interesting case is $\mathrm{diam}(X)=\pi$, though it is not necessary to assume so. If $\mathrm{diam}(X)<\pi$, then Theorem B(1,2) holds trivially. In fact, we can choose $\epsilon_X>0$ sufficiently small so that $\mathrm{diam}(Y)<\pi$ for all $Y$ with $d_{GH}(Y,X)\le\epsilon_X$. With this $\epsilon_X$, if $\widetilde{M}$ is $(C(X),\epsilon_X)$-stable at infinity, then $\widetilde{M}$ is $0$-Euclidean at infinity. As a result, $\pi_1(M)$ is finite by Proposition 1.9 in \cite{Pan2}.

For any Ricci limit space, its isometry group is always a Lie group \cite{CC2,CN}. This plays a crucial role in our proof.

\begin{thm}\label{limit_Lie}\cite{CC2,CN}
	For any $X\in\mathcal{M}_{cs}(n,0)$, $\mathrm{Isom}(X)$ is a Lie group.
\end{thm}

\begin{rem}\label{split_ortho}
	For $C(Y)\in \mathcal{M}(n,0)$, when $\mathrm{diam}(Y)=\pi$, $C(Y)$ contains at least one line.
	Thus $C(Y)$ splits isometrically as $\mathbb{R}^k\times C(Z)$, where $k\ge 1$ and $\mathrm{diam}(Z)<\pi$. Due to this splitting, the isometry group of the cross-section $Y$ also splits: 
	$$\mathrm{Isom}(Y)=O(k)\times \mathrm{Isom}(Z).$$
\end{rem}

We begin with some preparations on equivariant Gromov-Hausdorff topology \cite{FY}. For any isometric $G$-action on some metric space $Y$, we always assume that $G$-action is effectively and $G$ is a closed subgroup of $\mathrm{Isom}(Y)$.

\begin{defn}\label{eq_equiv_def}
	Let $Y_j$ be a compact metric space with isometric $G_j$-action ($j=1,2$). We say that $(Y_1,G_1)$ is equivalent to $(Y_2,G_2)$, if there is an isometry $F:Y_1\to Y_2$ and a group isomorphism $\psi:G_1\to G_2$ such that $F(g_1\cdot y_1)=\psi(g_1)\cdot F(y_1)$ for any $g_1\in G$ and $y_1\in Y_1$.
\end{defn}

\begin{rem}\label{eq_equiv_rem}
	For a compact metric space $Y$ with an isometric $G$-action, $G$ carries a natural bi-invariant metric coming from its action on $Y$:
	$$d_G(g_1,g_2)=\max_{y\in Y} d_Y(g_1y,g_2y).$$
	It is clear that in Definition \ref{eq_equiv_def}, $\psi: (G_1,d_{G_1})\to (G_2,d_{G_2})$ is an isometry.
\end{rem}

\begin{rem}
For any isometry $F:Y_1\to Y_2$, it induces a group isomorphism $C_F: \mathrm{Isom}(Y_1)\to\mathrm{Isom}(Y_2)$ by conjugation, that is, $C_F(g)=F\circ g\circ F^{-1}$. It is direct to check that $C_F$ satisfies $C_F(g)\cdot F(a)=F(g\cdot a)$ for any $a\in Y_1$ and any $g\in \mathrm{Isom}(Y_1)$. This implies that the group isomorphism $\psi$ in Definition \ref{eq_equiv_def} must be the conjugation map $C_F$. Indeed, consider the composition $\mathrm{id}=F^{-1}\circ F$, then $C_{F^{-1}}\circ \psi$ satisfies
$$C_{F^{-1}}\circ \psi(g)\cdot a=g\cdot a$$
for all $a\in X_1$ and $g\in G_1$. This shows that $\psi= (C_{F^{-1}})^{-1}=C_F$.
\end{rem}

One way to define the Gromov-Hausdorff distance is using approximation maps. Recall that a map $F:Y_1\to Y_2$ between two compact metric spaces is called an  \textit{$\epsilon$-GH approximation}, if\\
(1) $|d_{Y_1}(y,y')-d_{Y_2}(F(y),F(y'))|\le\epsilon$ for all $y,y'\in Y_1$,\\
(2) $F(Y_1)$ is $\epsilon$-dense in $Y_2$.

\begin{defn}
	Let $Y_1$ and $Y_2$ be two compact metric spaces with isometric $G_1$ and $G_2$ actions respectively. We say that
	$$d_{GH}((Y_1,G_1),(Y_2,G_2))\le\epsilon,$$
	if there a triple of maps $(F,\psi,\phi)$ such that\\
	(1) $F:Y_1\to Y_2$ is an $\epsilon$-GH approximation,\\
	(2) $\psi:G_1\to G_2$ is a map such that $d(F(gy),\psi(g)F(y))\le\epsilon$ for all $g\in G_1$ and $y\in Y_1$,\\
	(3) $\phi:G_2\to G_1$ is a map such that $d(F(\phi(g)y),gF(y))\le\epsilon$ for all $g\in G_2$ and $y\in Y_1$.
\end{defn}

\begin{rem}\label{tri}
Using triangle inequality, it is direct to verify that\\
(1) $\psi: G_1 \to G_2$ is an $5\epsilon$-GH approximation with respect to metrics in Remark \ref{eq_equiv_rem},\\
(2) $\psi:G_1\to G_2$ is almost a group homomorphism, that is,
$$d_{G_2}(\psi(gg'),\psi(g)\psi(g'))\le 5\epsilon$$
for all $g,g'\in G_1$.
\end{rem}

In general, $F$ and $\psi$ above may not be continuous. In \cite{MRW}, using center of mass technique, it was proved that when isometry groups are Lie groups and $\epsilon$ is sufficiently small, we can slightly modify the map $\psi:G_1\to G_2$ to a Lie group homomorphism.

\begin{prop}\cite{MRW}\label{good_map}
	Let $G$ be a Lie group with left-invariant Riemannian metric. Then there exists a constant $\epsilon(G)>0$ such that if $\psi:H\to G$ is a map from a Lie group $H$ to $G$ such that
	$$d(\psi(h_1h_2),\psi(h_1)\psi(h_2))\le\epsilon<\epsilon(G)$$
	for all $h_1,h_2\in H$, then there is a Lie group homomorphism $\bar{\psi}:H\to G$ with $d(\bar{\psi}(h),\psi(h))\le 2\epsilon$ for all $h\in H$.
\end{prop}

We formulate a definition for convenience.

\begin{defn}\label{def_good_map}
	Let $(Y_1,G_1)$ and $(Y_2,G_2)$ be two compact metric spaces with isometric actions, where $G_1$ and $G_2$ are Lie groups. We say that $\psi:G_1\to G_2$ is a $\delta$-approximated homomorphism, if\\
	(1) $\psi$ is a Lie group homomorphism, and\\
	(2) there is a $\delta$-GH approximation $f: Y_1\to Y_2$ such that
	$$d(f(gy),\psi(g) f(y))\le\delta$$
	for any $y\in Y_1$ and any $g\in G_1$.
\end{defn}

\begin{cor}\label{good_map_cor}\cite{MRW}
	Let $Y_i$ be a sequence of compact metric spaces with isometric Lie group $G_i$-actions. Suppose that
	$$(Y_i,G_i)\overset{GH}\longrightarrow (X,G).$$
	If $X$ is compact and $G$ is a Lie group, then there is a sequence of $\delta_i$-approximated homomorphisms $\phi_i:G_i\to G$ for some $\delta_i\to 0$. Moreover, $\phi_i$ are $\delta_i$-GH approximations between $G_i$ and $G$ with respect to the metrics in Remark \ref{eq_equiv_rem}.
\end{cor}

\begin{rem}
The maps $\phi_i$ in Corollary \ref{good_map_cor} may not be injective nor be surjective.
\end{rem}

For a fixed space $X\in\mathcal{M}_{cs}(n,0)$, when $Y\in \mathcal{M}_{cs}(n,0)$ is sufficiently close to $X$, we can apply Corollary \ref{good_map_cor} to project any isometric $H$-action on $Y$ to a subgroup of $\mathrm{Isom}(X)$ via the approximated homomorphism. We show that this map is canonical up to conjugations.

\begin{prop}\label{to_model}
	Given $X\in\mathcal{M}_{cs}(n,0)$, there exists positive constants $\epsilon_X,\zeta_X$ and a positive function $\delta(\epsilon)$ with $\lim\limits_{\epsilon\to 0}\delta(\epsilon)=0$ such that for any space $Y\in\mathcal{M}_{cs}(n,0)$ with $d_{GH}(X,Y)=\epsilon\le\epsilon_X$, the following properties holds:\\
	(1) There exists a $\delta(\epsilon)$-approximated homomorphism $$\psi:\mathrm{Isom}(Y)\to\mathrm{Isom}(X).$$
	(2) Let $H$ be a closed subgroup of $\mathrm{Isom}(Y)$ and $\phi:H\to\mathrm{Isom}(X)$ be a $\zeta_X$-approximation homomorphism, where $H$ is a closed subgroup of $\mathrm{Isom}(Y)$, then $\phi$ and $\psi|_H$ are conjugate.\\ 
	(3) $\zeta_X\ge 2\delta(\epsilon)$ for any $\epsilon\in(0,\epsilon_X]$.\\
	(4) $\psi|_{O(k)}$ is injective, where $O(k)$ is the subgroup of $\mathrm{Isom}(Y)$ in Remark \ref{split_ortho}.
\end{prop}

To prove Proposition \ref{to_model}, we need the following stability result on subgroups of a Lie group by Grove and Karcher \cite{GK}.

\begin{prop}\cite{GK}\label{conj}
	Let $\mu_1,\mu_2:H\to G$ be two Lie group homomorphisms of compact Lie group $H$ into the Lie group $G$ with a bi-invariant Riemannian metric. There exists $\epsilon(G)>0$ such that if $d(\mu_1(h),\mu_2(h))\le\epsilon(G)$ for all $h\in H$, then the subgroups $\mu_1(H)$ and $\mu_2(H)$ are conjugate in $G$.
\end{prop}

\begin{rem}
	Proposition \ref{conj} is stated with respect to a bi-invariant Riemannian metric $d_0$ on $G$. However, in practice we can apply this to any bi-invariant distance function $d_1$ on $G$. This follows from the fact that given any $\epsilon>0$, there is $\epsilon'>0$ such that $d_0(g_1,g_2)\le\epsilon$ whenever $d_1(g_1,g_2)\le\epsilon'$ for all $g_1,g_2\in G$.
\end{rem}

We prove Proposition \ref{to_model} as below.

\begin{proof}[Proof of Proposition \ref{to_model}]
	(1) follows directly from Corollary \ref{good_map_cor} and a standard contradiction argument. Suppose that (1) fails, then we would have a contradicting sequence $Y_i\in \mathcal{M}_{cs}(n,0)$ with $d_{GH}(Y_i,X)\to 0$. Passing to a subsequence if necessary, we obtain
	$$(Y_i,\mathrm{Isom}(Y_i))\overset{GH}\longrightarrow (X,G).$$
	By Theorem \ref{limit_Lie}, $G\subseteq \mathrm{Isom}(X)$ is a Lie group. Applying Corollary \ref{good_map_cor} to the above sequence, we result in the desired contradiction.
	
	(2): We may further shrink the constant $\epsilon_X$ that we just obtained from (1). To prove (2), we argue again by contradiction. Suppose that there are sequences $\epsilon_i,\zeta_i\to 0$ and a sequence $(Y_i,H_i)$ with the conditions below:\\
	(i) $d_{GH}(Y_i,X)\le \epsilon_i\to 0$;\\
	(ii) a sequence of $\delta_i$-approximated homomorphisms $\psi_i:H_i\to \mathrm{Isom}(X)$, where $\psi_i$ is given by (1) and $\delta_i=\delta(\epsilon_i)\to 0$;\\
	(iii) a sequence of $\zeta_i$-approximated homomorphisms $\phi_i:H_i\to \mathrm{Isom}(X)$ such that $(X,\psi_i(H_i))$ and $(X,\phi_i(H_i))$ are not equivalent.\\
	We prove that $(X,\psi_i(H_i))$ and $(X,\phi_i(H_i))$ are indeed equivalent for $i$ large, a contradiction, by showing that $\psi_i$ and $\phi_i$ are point-wise close up to an automorphism of $\mathrm{Isom}(X)$ (see Proposition \ref{conj}).
	
	By assumptions, we have
	$$d_{GH}((Y_i,H_i),(X,\psi_i(H_i))\le\delta_i\to 0,\quad d_{GH}((Y_i,H_i),(X,\phi_i(H_i))\le\zeta_i\to 0.$$
	Thus passing to some subsequences, all three sequences $(Y_i,H_i)$, $(X,\psi_i(H_i))$, and $(X,\phi(H_i))$ all converge to the same limit space $(X,H_\infty)$. Therefore, there is Lie group isomorphisms $\alpha_i$ and $\beta_i$, as conjugations in $(\mathrm{Isom}(X),d)$, such that
	$$\alpha_i\circ \psi_i(H_i)\overset{H}\to H_\infty,\quad \beta_i\circ \phi_i(H_i)\overset{H}\to H_\infty.$$
	
	We claim that $\alpha_i\circ \psi_i$ and $\beta_i\circ \phi_i$ are point-wise close, then the result would follow from Proposition \ref{conj}. Suppose that there is a sequence $h_i\in H_i$ and $d_0>0$ such that for all $i$
	$$d(\alpha_i\circ\psi_i(h_i),\beta_i\circ\phi_i(h_i))\ge d_0.$$
    With respect to the equivariant convergence of $(Y_i,H_i)$, we obtain
    $$(Y_i,H_i,h_i)\overset{GH}\longrightarrow (X,H_\infty,h_\infty).$$
    So do $(X,\psi_i(H_i),\psi_i(h_i))$ and $(X,\phi_i(H_i),\phi_i(h_i))$ Hausdorff converge to the same limit $(X,H_\infty,h_\infty)$.
    This implies that $\alpha_i\circ\psi_i(h_i)\to h_\infty$ and $\beta_i\circ \phi_i(h_i)\to h_\infty$. In other words,
    $$d(\alpha_i\circ\psi_i(h_i),\beta_i\circ\phi_i(h_i))\to 0.$$
    We apply Proposition \ref{conj} to conclude that (2) holds.
    
    (3) Since $\delta(\epsilon)\to 0$ as $\epsilon\to 0$, clearly we can shrink $\epsilon_X$ further so that $\delta(\epsilon)\le \zeta_X/2$ for all $\epsilon\in(0,\epsilon_X]$.
    
	(4) Recall that $O(k)$ acts on the Euclidean factor of $C(Y)$ as seen in Remark \ref{split_ortho}. Because any subgroup $H$ of $O(k)$ has displacement at least $1/20$ on $Y$, the image $\psi(H)$ has displacement at least $1/20-\delta(\epsilon)$ on $X$. Further shrink $\epsilon_X$ if necessary, we can assume that $\delta(\epsilon)<1/40$ for all $\epsilon\in (0,\epsilon_X]$. This guarantees that $\psi(H)$ can not be trivial.
\end{proof}

With Proposition \ref{to_model}, we define the notion of $X$-mark, which provides a way to compare isometric actions on alike spaces.

\begin{defn}
	Given $X\in\mathcal{M}_{cs}(n,0)$ and the corresponding $\epsilon_X$ in Proposition \ref{to_model}. For a space $(Y,H)$ with $d_{GH}(X,Y)\le\epsilon_X$, we call $(X,\psi(H))$ in Proposition \ref{to_model} as the \textit{$X$-mark} of $(Y,H)$. Note that Proposition \ref{to_model}(2) assures that the $X$-mark of $(Y,H)$ is unique in the equivariant Gromov-Hausdorff topology.
\end{defn}

For the remaining of this section, we always denote $\epsilon_X>0$ as the constants in Proposition \ref{to_model}. Next we show that the notion of $X$-mark is compatible with equivariant Gromov-Hausdorff convergence, in the sense of that the convergence of spaces implies the convergence of $X$-marks. 

\begin{thm}\label{mark_conv}
	Let $(Y_i,H_i)$ be a sequence of spaces with $Y_i\in \mathcal{M}_{sc}(n,0)$ and $d_{GH}(Y_i,X)\le \epsilon_X$, the constant in Proposition \ref{to_model}. Suppose that
	$$(Y_i,H_i)\overset{GH}\longrightarrow (Y,H)$$
	and each $(Y_i,H_i)$ has $X$-mark $(X,G_i)$. Then
	$$(X,G_i)\overset{GH}\longrightarrow (X,G),$$
	where $(X,G)$ is the $X$-mark of $(Y,H)$.
\end{thm}

\begin{proof}
	We put $\epsilon_i=d_{GH}(Y_i,X)$ and $\delta_i=\delta(\epsilon_i)$ as given by Proposition \ref{to_model}. We first note that the statement is trivial when $Y=X$. In fact, when $Y=X$ it is clear that the limit space $(X,H)$ has itself as its $X$-mark. By Proposition \ref{to_model}, for each $i$ there is $\delta_i$-approximated homomorphism $\psi_i:H_i\to G_i=\phi_i(H_i)$ with $\delta_i\to 0$. This implies that the sequences $\{(Y_i,H_i)\}$ and $\{(X_i,G_i)\}$ share the same limit $(X,H)$. 
	
	For the remaining proof, we assume that limit space $Y$ is not isometric to $X$, that is, $\epsilon_i\ge \epsilon_0>0$ for some $\epsilon_0$. For each $(Y_i,H_i)$, there is a $\delta_i$-approximated Lie group homomorphism $\psi_i:H_i\to \mathrm{Isom}(X)$ with $\psi_i(H_i)=G_i$. Also let $\psi: H\to \mathrm{Isom}(X)$ with $\psi (H)=G$ be a $\delta$-approximated Lie group homomorphism that marks the limit space $(Y,H)$. From Corollary \ref{good_map_cor} and the convergence
	$$(Y_i,H_i)\overset{GH}\longrightarrow (Y,H),$$
	we know that for each $i$ large there is an $\eta_i$-approximated Lie group homomorphism
	$$f_i: H_i\to H$$
	for some $\eta_i\to 0$. Passing to a subsequence if necessary, we obtain
	$$(X,G_i)\overset{GH}\longrightarrow (X,G_\infty),$$
	or equivalently, the Hausdorff convergence
	$$G_i\overset{H}\longrightarrow G_\infty$$
	in $\mathrm{Isom}(X)$. We need to show that $(X,G_\infty)$ is equivalent to $(X,G)$.
	
	Note that $\psi\circ f_i$ gives a $(\delta+\eta_i)$-approximated Lie group homomorphism from $H_i$ to $\psi\circ f_i(H_i)\subset \mathrm{Isom}(X)$. For $i$ large, $\psi\circ f_i$ is $\zeta_X$-approximated because $\zeta_X\ge 2\delta$ and $\eta_i\to 0$, where $\zeta_X$ is the constant in Proposition \ref{to_model}. We see that $(X,G_i)$ is equivalent to $(X,\psi\circ f_i(H_i))$ for $i$ large from Proposition \ref{to_model}(2,3).

	For simplicity, we introduce a notation here: for two subgroups $K_1$ and $K_2$ in $\mathrm{Isom}(X)$, we write $K_1\sim K_2$ if they are conjugate in $\mathrm{Isom}(X)$. From the discussion above, we have
	$$G_i \sim \psi\circ f_i(H_i)\subseteq \psi(H)=G.$$
	Recall that $f_i$ is an $\eta_i$-approximation from $H_i$ to $H$. Consequently, the image $f_i(H_i)$ is $\eta_i$-dense in $H$. Through $\psi:H\to \psi(H)$, $\psi\circ f_i(H_i)$ is $\eta'_i$-dense in $\psi(H)$ for some $\eta'_i\to 0$. In other words, we have Hausdorff convergence
	$$\psi\circ f_i (H_i)\overset{H} \longrightarrow \psi(H)=G.$$
	Together with
	$$\psi\circ f_i(H_i)\sim G_i\overset{H}\longrightarrow G,$$
	we conclude that $G_\infty \sim G$ and complete the proof.
\end{proof}

With Theorem \ref{mark_conv}, we show that given $(X,G)$, there is a uniform gap between any space with $X$-mark $(X,G)$ and any space with higher dimensional $X$-mark.  

\begin{thm}\label{isom_stable}
	Let $(X,G)$ be a space with $X\in \mathcal{M}_{cs}(n,0)$. There exists a constant $\eta>0$, depending on $(X,G)$, such that the following holds.
	
	For any two spaces $(Y_j,K_j)$ $(j=1,2)$ satisfying\\
	(1) $Y_j\in\mathcal{M}_{cs}(n,0)$ and $d_{GH}(Y_j,X)\le \epsilon_X$, the constant in Proposition \ref{to_model},\\
	(2) $(Y_1,K_1)$ has $X$-mark $(X,G)$,\\
	(3) $(Y_2,K_2)$ has $X$-mark $(X,H)$ with $\dim(H)\ge \dim(G)$.\\
	If $$d_{GH}((Y_1,K_1),(Y_2,K_2))\le \eta,$$
	then $(X,H)$ is equivalent to $(X,G)$.
\end{thm}

Theorem \ref{isom_stable} generalizes the theorem below in \cite{Pan2}, which considers group actions on a fixed space $X$. Although the statement of Proposition 3.1 in \cite{Pan2} only covers the case $X$ as any compact Riemannian manifold, its proof actually works for any space $X\in\mathcal{M}_{cs}(n,0)$ through verbatim since the proof only requires that $\mathrm{Isom}(X)$ is a compact Lie group (see Remark 3.5 in \cite{Pan2}).

\begin{thm}\label{isom_stable_single}\cite{Pan2}
	Let $(X,G)$ be a space with $X\in \mathcal{M}_{cs}(n,0)$. There exists a constant $\eta>0$, depending on $(X,G)$, such that the following holds.
	
	For any isometric $H$-action on $X$ with $\dim(H)\ge \dim(G)$, if
	$$d_{GH}((X,G),(X,H))\le\eta,$$
	then $(X,H)$ is equivalent to $(X,G)$.
\end{thm}

\begin{proof}[Proof of Theorem \ref{isom_stable}]
	We argue by contradiction. Suppose that the statement fails, then we would have two sequences $\{(Y_{i,j},K_{i,j})\}_i$ $(j=1,2)$ satisfying the conditions below.\\
	(1) $Y_{i,j}\in \mathcal{M}_{cs}(n,0)$ and $d_{GH}(Y_{i,j},X)\le\epsilon_X$ for all $i,j$;\\
	(2) $(Y_{i,1},K_{i,1})$ has $X$-mark $(X,G)$ for each $i$;\\
	(3) $(Y_{i,2},K_{i,2})$ has $X$-mark $(X,H_i)$ with $\dim(H_i)\ge \dim(G)$, but $(X,H_i)$ is not equivalent to $(X,G)$;\\
	(4) $d_{GH}((Y_{i,1},K_{i,1}),(Y_{i,2},K_{i,2}))\to 0$.\\
	Passing to a subsequence, we obtain the equivariant convergence of two sequences with the same limit:
	$$(Y_{i,1},K_{i,1})\overset{GH}\longrightarrow (Y,K),$$
	$$(Y_{i,2},K_{i,2})\overset{GH}\longrightarrow (Y,K).$$
	By conditions (1,2) and Theorem \ref{mark_conv}, the limit $(Y,K)$ has $X$-mark $(X,G)$. Apply Theorem \ref{mark_conv} again with condition (3), we conclude that
	$$(X,H_i)\overset{GH}\longrightarrow (X,G).$$
	Then the desired contradiction follows from Theorem \ref{isom_stable_single} and the assumption that $\dim(H_i)\ge \dim(G)$.
\end{proof}

For our application, we will actually use a slightly different version of Theorem \ref{isom_stable} (see Proposition \ref{gap_isotropy}).

\section{Isometric actions on metric cones}

This section serves as preparations for Theorem \ref{main_omega}. We study some special isometric actions on metric cones.

Recall that for a metric cone $(C(Z),z)\in \mathcal{M}(n,0)$ with vertex $z$, it splits isometrically as $$(C(Z),z)=(\mathbb{R}^k \times C(Z'),(0,z')),$$
where $\mathrm{diam}(Z')<\pi$. As a result, the isometry group of $C(Z)$ also splits:
$$\mathrm{Isom}(C(Z))=\mathrm{Isom}(\mathbb{R}^k)\times \mathrm{Isom}(C(Z'))=(O(k)\ltimes \mathbb{R}^k)\times \mathrm{Isom}(Z').$$
With this splitting, we write every element in $\mathrm{Isom}(C(Z))$ as $(A,v,\alpha)$. Also, elements in $\mathrm{Isom}(Z)=O(k)\times \mathrm{Isom}(Z'))$ can be written as $(A,0,\alpha)$.

Property \textit{(P)} was introduced in \cite{Pan2} for abelian group actions on $\mathbb{R}^k$. We naturally extend it here to the nilpotent group actions on metric cones.

\begin{defn}\label{def_linear_P}
	Let $(C(Z),z)\in\mathcal{M}(n,0)$ be a metric space with vertex $z$. We say that a nilpotent isometric $G$-action on $(C(Z),z)$ satisfies \textit{property (P)}, if\\ \textit{(P)} for any $(A,v,\alpha)\in G$, $(A,0,\alpha)$ is also in $G$.
\end{defn}

We also introduce the notion of property \textit{(Q)} for $G$-action on a metric cone $C(Z)$ with $d_{GH}(Z,X)\le\epsilon_X$, where $X\in \mathcal{M}_{cs}(n,0)$ is a fixed model space and $\epsilon_X$ is the constant in Proposition \ref{to_model}. Let $(C(Z),z)\in\mathcal{M}(n,0)$ be a metric cone with vertex $z$ and isometric $G$-action, we denote $(C(Z),z,G_\infty)$ as the equivariant tangent cone of $(C(Z),z,G)$ at infinity: $$(j^{-1}C(Z),z,G)\overset{GH}\longrightarrow(C(Z),z,G_\infty).$$

\begin{defn}\label{def_linear_Q}
	Let $(C(Z),z)\in\mathcal{M}(n,0)$ be a metric cone with vertex $z$ and $d_{GH}(X,Z)\le\epsilon_X$. We say that a nilpotent isometric $G$-action on $(C(Z),z)$ satisfies \textit{property (Q)} (with respect to $C(X)$), if\\ \textit{(Q)} $(Z,\mathrm{Iso}_zG)$ and $(Z,\mathrm{Iso}_z G_\infty)$ have the same $X$-mark.
\end{defn}

\begin{prop}\label{lin_P_decom}
	$(C(Z),z,G)$ satisfies \textit{property (P)} if and only if $G$-action fulfills the conditions below:\\
	(1) $G$ admits decomposition $G=\mathrm{Iso}_z G\times \mathbb{R}^{l}\times \mathbb{Z}^{m}$.\\
	(2) Any element in the subgroup $\{e\}\times \mathbb{R}^{l}\times \mathbb{Z}^{m}$ is a translation in the Euclidean factor of $C(Z)$.
\end{prop}

It is clear that (1) and (2) implies property \textit{(P)}. Also, Proposition \ref{lin_P_decom} shows that property \textit{(P)} implies property \textit{(Q)}. For abelian $G$-action, Proposition \ref{lin_P_decom} is clear; for nilpotent one, we need the lemma below. 

\begin{lem}\label{commute_E}
	Let $G$ be a nilpotent subgroup of $\mathrm{Isom}(\mathbb{R}^k)$. Let $(A,x)$ and $(B,y)$ be two elements of $G$. Then $(A,x)$ and $(B,y)$ commutes if and only if $A$ and $B$ commutes.
\end{lem}

\begin{proof}
	The proof is linear algebra. We include the proof for readers' convenience.
	
	By direct calculation, we have
	\begin{align*}
		[(A,x),(B,y)]=&(A,x)\cdot (B,y)\cdot (A,x)^{-1}\cdot (B,y)^{-1}\\
		=& ([A,B],-[A,B]y-ABA^{-1}x+Ay+x).
	\end{align*}
	Clearly if $(A,x)$ and $(B,y)$ commutes, so does $A$ and $B$.
	
	Conversely, if $A$ and $B$ commutes, then
	$$[(A,x),(B,y)]=(I,-y-Bx+Ay+x),$$
	which is a translation. We denote this vector as $w=-y-Bx+Ay+x$.
	
	Since $G$ is nilpotent, after $l$ times of commutator calculation, we result in
	$$[(A,x),[...,[(A,x),(I,w)]]]=(I,0).$$
	It is easy to verify that the left hand side equals to $(I,(A-I)^l w)$. Thus $$(A-I)^l w=0$$ for some $l$. With the fact that $A\in O(k)$, we have
	\begin{center}
		$(A-I)^l w=0\ $ if and only if $\ (A-I)w=0$.
	\end{center}
	Therefore, $Aw=w$. Similarly, we have $Bw=w$. Since $A$ and $B$ commutes, they share the same eigen-space decomposition. We define a subspace $E$ as
	\begin{center}
		$E=\{v\in\mathbb{R}^n\ |\ Av=v=Bv\}$
	\end{center}
	and decompose $\mathbb{R}^n$ as $E+E^\bot$, where $E^\bot$ is the orthogonal complement of $E$. We write $x=x_1+x^\bot$ and $y=y_1+y^\bot$ according to this decomposition. Then $w$ can be written as $Ay^\bot -y^\bot +x^\bot-Bx^\bot$, which is a vector in $E^\bot$. Together with the fact that $w\in E$, we conclude that $w=0$ and complete the proof.
\end{proof}

\begin{proof}[Proof of Proposition \ref{lin_P_decom}]
	Let $K=\mathrm{Iso}_zG$ be the isotropy subgroup at $z$, that is, the subgroup consisting all elements of $G$ with form $(A,0,\alpha)$. Let $T$ be the subgroup of $G$ consisting all elements of $G$ with form $(I,v,e)$. It is clear that $K\cap T=\{(I,0,e)\}$.
	
	It is not difficult to see that $G=TK$. In fact, for any $(A,v,\alpha)\in G$, by assumptions, we have $(A,0,\alpha)\in G$. Thus
	$$(A,v,\alpha)\cdot(A^{-1},0,\alpha^{-1})=(I,v,e)\in G.$$
	Hence we can write $(A,v,\alpha)$ as a product $(I,v,e)\cdot (A,0,\alpha)$.
	
	Also, note that any two elements $(A,0,\alpha)\in K$ and $(I,v,e)\in T$ must commute due to Lemma \ref{commute_E}. This shows that $G$ is isomorphic to $K\times T$. To complete the proof, recall that $T$ is a subgroup consisting of translations, thus $T=\mathbb{R}^l\times \mathbb{Z}^m$ for some integers $l$ and $m$.  
\end{proof}

\begin{lem}\label{lin_P_pass}
	Let $(C(Z),z)$ be a metric cone with vertex $z$ and isometric $G$-action. We consider the equivariant tangent cone at $y$, or at infinity ($r_i\to\infty$):
	$$(r_iC(Z),z,G)\overset{GH}\longrightarrow(C(Z),z,G_z),$$
	$$(r_i^{-1}C(Z),z,G)\overset{GH}\longrightarrow(C(Z),z,G_\infty).$$
	Then both $(C(Z),z,G_z)$ and $(C(Z),z,G_\infty)$ satisfy property \textit{(P)}.
\end{lem}

\begin{proof}
	Let $K$ be the subgroup of $G$ fixing $x$. It is clear that
	$$(r_iC(Z),z,G)\overset{GH}\longrightarrow (C(Z),z,K\times \mathbb{R}^l),$$
	where $l$ is the dimension of the orbit $G\cdot z$ in the Euclidean factor of $C(Z)$, and $\{e\}\times \mathbb{R}^l$ acts as translations.
	
	Next we check that $(C(Z),z,G_\infty)$ satisfies property \textit{(P)}. Let $(A,v,\alpha)$ be an element in $G_\infty$ with $v\not= 0$. Due to the convergence, this means that there are a sequence $(A_i,t_iv_i,\alpha_i)\in G$ with $A_i\to A$, $t_i/r_i\to 1$, $v_i\to v$ and $\alpha_i\to\alpha$. For each fixed integer $j$,
	$$(r_i^{-1}C(Z),z,(A_j,t_jv_j,\alpha_j))\overset{GH}\longrightarrow(\mathbb{R}^k,0,(A_j,0,\alpha_j)).$$
	Thus $(A_j,0,\alpha_j)\in G_\infty$. Since $A_j\to A$ and $\alpha_j\to\alpha$ as $j\to\infty$ and $G_\infty$ is closed, we conclude that $(A,0,\alpha)\in G_\infty$. Therefore, $(C(Z),z,G_\infty)$ satisfies property \textit{(P)}.
\end{proof}

\begin{lem}\label{lin_Q_prop}
	Let $(C(Z),z)=(\mathbb{R}^k\times C(Z'),(0,z'))\in\mathcal{M}(n,0)$ be a metric cone with vertex $z$, $\mathrm{diam}(Z')<\pi$, and $d_{GH}(X,Z)\le\epsilon_X$. Let $\psi:\mathrm{Isom}(Z)\to \mathrm{Isom}(X)$ be the map in Proposition \ref{to_model}. Suppose that $(C(Z),z,G)$ satisfies property \textit{(Q)}. Then for any $(A,v,\alpha)\in G$, $G$ contains $(A,0,\beta)$ with $\psi |_{\mathrm{Isom}(Z')}(\alpha\beta^{-1})=e$. In particular, $(\mathbb{R}^k,0,p_Z(G))$ satisfies property \textit{(P)}, where $p_Z:\mathrm{Isom}(C(Z))\to\mathrm{Isom}(\mathbb{R}^k)$ is the projection.
\end{lem}

\begin{proof}
	It is clear that $\mathrm{Iso}_zG$ is a subgroup of $\mathrm{Iso}_z G_\infty$. It follows from the definition that $(Z,\mathrm{Iso}_zG)$ and $(Z,\mathrm{Iso}_z G_\infty)$ have the same $X$-mark, thus
	$$\psi(\mathrm{Iso}_z G_\infty)=\psi (\mathrm{Iso}_zG).$$
	For any $(A,v,\alpha) \in G$, note that $(A,0,\alpha)$ must be in $G_\infty$. Then there is $(B,0,\beta)\in\mathrm{Iso}_z G$ such that
	$$\psi(A,0,\alpha)=\psi(B,0,\beta).$$
	That is,
	$$\psi(AB^{-1},0,\alpha\beta^{-1})=e.$$
	By Proposition \ref{to_model}(3), $\psi|_{O(k)}$ is injective. It follows that $A=B$. Therefore, for any $(A,0,\alpha)\in G$ we have $(A,0,\beta)\in \mathrm{Iso}_zG$ with $\psi(I,0,\alpha\beta^{-1})=e$.
\end{proof}

\begin{lem}\label{lin_Q_conv}
	Let $(C(Z_i),z_i,G_i)$ be a sequence of spaces with $d_{GH}(Z_i,X)\le\epsilon_X$ and property \textit{(Q)}. Suppose that
	$$(C(Z_i),z_i,G_i)\overset{GH}\longrightarrow (C(Z),z,G).$$
	Then the limit space $(C(Z),z,G)$ also satisfies property \textit{(Q)}.
	
	If in addition each $(Z_i,\mathrm{Iso}_{z_i}G_i)$ has the $X$-mark $(X,K_i)$, then $$(X,K_i)\overset{GH}\longrightarrow (X,K),$$
	where $(X,K)$ is the $X$-mark of $(Z,\mathrm{Iso}_z G)$.
\end{lem}

\begin{proof}
	We consider the following convergent sequences:
	$$(Z_i,\mathrm{Iso}_{z_i} G_i)\overset{GH}\longrightarrow (Z,H),$$
	$$(Z_i,\mathrm{Iso}_{z_i} (G_i)_\infty)\overset{GH}\longrightarrow(Z,L).$$
	It is clear that
	$$H\subseteq \mathrm{Iso}_zG\subseteq\mathrm{Iso}_zG_\infty \subseteq L.$$
	Since $(Z_i,\mathrm{Iso}_{z_i} G_i)$ and $(Z_i,\mathrm{Iso}_{z_i} (G_i)_\infty)$ have the same $X$-mark for each $i$, by Theorem \ref{mark_conv}, we conclude that $(Z,H)$ and $(Z,L)$ also share the same $X$-mark, which is also the $X$-mark of $(Z,\mathrm{Iso}_zG)$ and $(Z,\mathrm{Iso}_zG_\infty)$. 
\end{proof}

\section{Proof of equivariant stability at infinity}

With the preparations in Section 2, we prove our main goal, equivariant stability at infinity, in this section.

\begin{thm}\label{main_omega'}
	Let $X\in\mathcal{M}_{cs}(n,0)$ be a metric space and $\epsilon_X>0$ be the constant in Proposition
	\ref{to_model}. Then the following statement holds.
	
	Let $(M,x)$ be an open $n$-manifold with $\mathrm{Ric}\ge 0$. Suppose that $\pi_1(M)$ is nilpotent and $\widetilde{M}$ is $(C(X),\epsilon_X)$-stable at infinity. Then there exist an integer $l$ and a closed subgroup $K$ of $\mathrm{Isom}(X)$ such that any equivariant tangent cone of $(\widetilde{M},\pi_1(M,x))$ at infinity $(C(Z),z,G)$ satisfies that\\
	(1) property \textit{(Q)} holds,\\
	(2) $(Z,\mathrm{Iso}_z G)$ has $X$-mark $(X,K)$,\\
	(3) the orbit $G\cdot z$ is an $l$-dimensional Euclidean subspace in $C(Z)$.
\end{thm}

We compare Theorem \ref{main_omega'} with the result below, which regards the case that $\widetilde{M}$ has unique tangent cone at infinity as a metric cone (see Remark 4.15 in \cite{Pan2}).

\begin{thm}\label{main_omega_unique}
	Let $(M,x)$ be an open $n$-manifold with $\mathrm{Ric}\ge 0$. Suppose that $\pi_1(M)$ is nilpotent and $\widetilde{M}$ has unique tangent cone at infinity $(C(Z),z)$ as a metric cone with vertex $z$. Then there exist an integer $l$ and a closed subgroup $K$ of $\mathrm{Isom}(X)$ such that for any equivariant tangent cone of $(\widetilde{M},\pi_1(M,x))$ at infinity $(C(Z),z,G)$, property \textit{(P)} holds with $G=K\times \mathbb{R}^l$ (see Lemma \ref{lin_P_decom}).
\end{thm}

We have seen in Section 2 that property \textit{(Q)} is weaker than \textit{(P)} (see Proposition \ref{lin_P_decom} and Lemma \ref{lin_Q_prop}). Therefore, compared with Theorem \ref{main_omega_unique}, the conclusion in Theorem \ref{main_omega'} allows more flexibility. More technically speaking, this comes from the fact that the map $\psi$ in Proposition \ref{to_model} may not be injective. If one assume that $\widetilde{M}$ has Euclidean volume growth and is $(C(X),\epsilon_X)$-stable at infinity, then $X\in\mathcal{M}_{cs}(n,0,v)$ for some $v>0$; in this case, one can apply Theorem 0.8 in \cite{PR} to show that $\psi$ in Proposition \ref{to_model} is injective, which will improve Theorem \ref{main_omega'}(1) from property \textit{(Q)} to \textit{(P)}.

We first follow the idea in Section 1 to establish an equivariant Gromov-Hausdorff gap for metric cones with property \textit{(Q)}.

\begin{lem}\label{gap_isotropy}
	Let $K$ be an isometric action on $X\in\mathcal{M}_{cs}(n,0)$. Then there exists $\eta(X,K)>0$, such that the following holds.	
	
	For any two metric cones $(C(Z_j),z_j,G_j)$ with the conditions below\\
	(1) $d_{GH}(X,Z_j)\le \epsilon_X$, the constant in Proposition \ref{to_model} ($j=1,2$),\\
	(2) $(C(Z_j),z_j,G_j)$ satisfies property \textit{(Q)} ($j=1,2$),\\
	(3) $(Z_1,\mathrm{Iso}_{z_1} G_1)$ has $X$-mark $(X,K)$,\\
	(4) $(Z_2,\mathrm{Iso}_{z_2} G_2)$ has $X$-mark $(X,H)$ with $\dim(H)\ge\dim (K)$ and $(X,H)$ not equivalent to $(X,K)$,\\
	then
	$$d_{GH}((C(Z_1),z_1,G_1),(C(Z_2),z_2,G_2))\ge\eta.$$
\end{lem}

\begin{proof}
	Suppose that there are two sequences of metric cones: $$\{(C(Z_{ij}),z_{ij},G_{ij})\}_i\ \ (j=1,2)$$ such that for all $i$,\\
	(1) $d_{GH}(X,Z_{ij})\le \epsilon_X$;\\
	(2) $(C(Z_{ij}),z_{ij},G_{ij})$ satisfies property \textit{(Q)};\\
	(3) $(Z_{i1},\mathrm{Iso}_{z_{i1}} G_{i1})$ has $X$-mark $(X,K)$;\\
	(4) $(Z_{i2},\mathrm{Iso}_{z_{i2}} G_{i2})$ has $X$-mark $(X,H_i)$, with $\dim(H_i)\ge\dim (K)$ and $(X,H_i)$ not equivalent to $(X,K)$;\\
	(5) $d_{GH}((C(Z_{i1}),z_{i1},G_{i1}),(C(Z_{i2}),z_{i2},G_{i2}))\to 0$ as $i\to\infty$.\\
	After passing to some subsequences, this gives convergence
	$$(C(Z_{i1}),z_{i1},G_{i1})\overset{GH}\longrightarrow(C(Z),z,G),$$
	$$(C(Z_{i2}),z_{i2},G_{i2})\overset{GH}\longrightarrow(C(Z),z,G).$$
	Since each $(C(Z_{ij}),z_{ij},G_{ij})$ satisfies property \textit{(Q)}, by Lemma \ref{lin_Q_conv}, the limit space $(C(Z),z,G)$ also satisfies property \textit{(Q)}. Moreover, $(Z,\mathrm{Iso}_z G)$ has $X$-mark $(X,K)$ due to condition (3). Applying Lemma \ref{lin_Q_conv} again, we also see that
	$$(X,H_i)\overset{GH}\longrightarrow (X,K)$$
	with $\dim(H_i)\ge \dim(K)$ and $(X,H_i)$ not equivalent to $(X,K)$, which contradicts Theorem \ref{isom_stable_single}.
\end{proof}

We prove a key lemma to Theorem \ref{main_omega'} using the critical rescaling argument.

\begin{lem}\label{omega_isotropy}
	Let $M$ be an open $n$-manifold with the assumptions in Theorem \ref{main_omega'}. Then for any two spaces $(C(Z_j),z_j,G_j)\in (\widetilde{M},\Gamma)$ with property \textit{(Q)} $(j=1,2)$,\\ $(Z_1,\mathrm{Iso}_{z_1} G_1)$ and $(Z_2,\mathrm{Iso}_{z_2} G_2)$ have the same $X$-mark.
\end{lem}

\begin{defn}
	For a compact Lie group $K$, we define $$D(K)=(\dim K,\#K/K_0).$$ For two compact Lie groups $K$ and $H$, with $D(K)=(l_1,l_2)$ and $D(H)=(m_1,m_2)$, we say that $D(K)<D(H)$, if $l_1<m_1$, or if $l_1=m_1$ and $l_2<m_2$. We say that $D(K)\le D(H)$, if $D(K)=D(H)$ or $D(K)<D(H)$.
\end{defn}

\begin{proof}[Proof of Lemma \ref{omega_isotropy}]
	We argue by contradiction. Suppose that there are two spaces $$(C(Z_j),z_j,G_j)\in \Omega(\widetilde{M},\Gamma)\ \ (j=1,2)$$ with property \textit{(Q)} but
	$(Z_1,\mathrm{Iso}_{z_1} G_1)$ and $(Z_2,\mathrm{Iso}_{z_2} G_2)$ having different $X$-marks. We also choose $(C(Z_1),z_1,G_1)$ so that its isotropy subgroup at $z_1$, $\mathrm{Iso}_{z_1}G_1$, has the minimal $D(K_1)$ among all spaces in $\Omega(\widetilde{M},\Gamma)$ with property \textit{(Q)}. We will derive a contradiction by using the critical rescaling argument and Lemma \ref{gap_isotropy}.
	
	Let $(X,K_j)$ be the $X$-mark of $(Z_j,\mathrm{Iso}_{z_j}G_j)$.
	Let $r_i\to\infty$ and $s_i\to\infty$ be two sequences such that
	$$(r^{-1}_i\widetilde{M},\tilde{x},\Gamma)\overset{GH}\longrightarrow(C(Z_1),z_1,G_1),$$
	$$(s^{-1}_i\widetilde{M},\tilde{x},\Gamma)\overset{GH}\longrightarrow(C(Z_2),z_2,G_2),$$
	Passing to a subsequence if necessary, we assume that $t_i:=(s^{-1}_i)/(r^{-1}_i)\to\infty$. Put $$(N_i,q_i,\Gamma_i)=(r^{-1}_i\widetilde{M},\tilde{x},\Gamma),$$
	then we have
	$$(N_i,q_i,\Gamma_i)\overset{GH}\longrightarrow(C(Z_1),z_1,G_1),$$
	$$(t_iN_i,q_i,\Gamma_i)\overset{GH}\longrightarrow(C(Z_2),z_2,G_2).$$
	By assumptions, $(C(Z_j),z_j,G_j)$ satisfies property \textit{(Q)} $(j=1,2)$; also, $D(K_1)=(m_1,m_2)\le D(K_2)$, and $(X,K_1)$ is not equivalent to $(X,K_2)$.
	
	We choose $\eta=\eta(X,K_1)>0$ as follows: by Lemma \ref{gap_isotropy}, there is $\eta>0$ such that for any $(C(Y_j),y_j,H_j)\in \Omega(\widetilde{M},\Gamma)$ $(j=1,2)$ satisfying\\
	(1) $(C(Y_j),y_i,H_j)$ satisfies property \textit{(Q)},\\
	(2) $(Y_1,\mathrm{Iso}_{z_1}H_1)$ has $X$-mark $(X,K_1)$,\\
	(3) $(Y_2,\mathrm{Iso}_{z_2}H_2)$ has $X$-mark $(X,K)$ with $\dim(K)\ge\dim(K_1)$,\\
	(4) $d_{GH}((C(Y_1),y_1,H_1),(C(Y_2),y_2,H_2))\le\eta$,\\then $(X,K)$ is equivalent to $(X,K_1)$.
	
	For each $i$, we define a set of scales
	\begin{align*}
		L_i:=\{\  1\le l\le t_i\ |\ & d_{GH}((l{N}_i,q_i,\Gamma_i),(C(Y),y,H))\le \eta/3 \\
		& \text{for some space $(C(Y),y,H)\in \Omega(\widetilde{M},\Gamma)$}\\
		& \text{such that $(C(Y),y,H)$ satisfies property \textit{(Q)};}\\
		& \text{moreover, } D(K)\ge(m_1,m_2) \text{ but } (X,K)\\
		& \text{is not equivalent to } (X,K_1), \text{ where}\\
		& (X,K) \text { is the $X$-mark of } (Y,\mathrm{Iso}_yH)\}.
	\end{align*}

    Since $t_i\in L_i$ for all $i$ large, $L_i$ is non-empty. We choose $l_i\in L_i$ so that $\inf L_i\le l_i \le \inf L_i+1/i$.
    
    \textbf{Claim 1:} $l_i\to\infty$. Suppose that $l_i\to B<\infty$ for some subsequence, then for this subsequence,
    $$(l_iN_i,q_i,\Gamma_i)\overset{GH}\longrightarrow(B\cdot C(Z_1),z_1,G_1).$$
    By the fact that $l_i\in L_i$ and the above convergence, we know that there is some space $(C(Y),y,H)\in\Omega(\widetilde{M},\Gamma)$ with the properties below:\\
    (1) $(C(Y),y,H)$ satisfies property \textit{(Q)},\\
    (2) $D(K)\ge(m_1,m_2)$ but $(X,K)$ is not equivalent to $(X,K_1)$, where $(X,K)$ is the $X$-mark of $(Y,\mathrm{Iso}_y(H))$,\\
    (3) $d_{GH}((B\cdot C(Z_1),z_1,G_1),(C(Y),y,H))\le\eta/2$.\\
    Since $(C(Z_1),z_1,G_1)$ satisfies property \textit{(Q)}, so does $(B\cdot C(Z_1),z_1,G_1)$. By Lemma \ref{gap_isotropy} and the choice of $\eta$, we derive a contradiction to the condition (2) above. We have verified Claim 1.
    
    Passing to a subsequence if necessary, we consider the convergence
    $$(l_iN_i,q_i,\Gamma_i)\overset{GH}\longrightarrow(C(Z'),z',G').$$
    We draw a contradiction by ruling out all the possibilities of $G'$-action. Let $(X,K')$ be the $X$-mark of $(Z',\mathrm{Iso}_{z'}G')$.
    
    \textbf{Claim 2:} $D(K')\ge (m_1,m_2)$. In fact, if $D(K')<(m_1,m_2)$, we pass to the equivariant tangent cone of $(C(Z'),z',G')$ at $z'$:
    $$(jC(Z'),z',G')\overset{GH}\longrightarrow (C(Z'),z',G'_{z'}).$$
    The limit space $(C(Z'),z',G'_{z'})$ satisfies property \textit{(P)}, thus property \textit{(Q)}. Since $\mathrm{Iso}_{z'}G'_{z'}=\mathrm{Iso}_{z'}G'$, $(Z',\mathrm{Iso}_{z'}G'_{z'})$ has $X$-mark $(X,K')$ with $D(K')<(m_1,m_2)$. We know that this can not happen since we chose $(C(Z_1),z_1,G_1)$ so that $K_1$, the $X$-mark of $\mathrm{Iso}_{z_1} G_1$, has the minimal $D(K_1)$ among all spaces in $\Omega(\widetilde{M},\Gamma)$ with property \textit{(Q)}.
    
    \textbf{Claim 3:} $(C(Z'),z',G')$ satisfies property \textit{(Q)} and $D(K')=(m_1,m_2)$. We already know that $D(K')\ge (m_1,m_2)$ from Claim 2. Suppose that Claim 3 fails, then after we passing to the equivariant tangent cone of $(C(Z'),z',G')$ at infinity:
    $$(j^{-1}C(Z'),z',G')\overset{GH}\longrightarrow(C(Z'),z',G'_\infty),$$
    the limit space $(C(Z'),z',G'_\infty)$ must satisfies $D(K'_\infty)>(m_1,m_2)$, where $(X,K'_\infty)$ is the $X$-mark of $(Z',\mathrm{Iso}_{z'}G'_\infty)$. Also,
    $(C(Z'),z',G'_\infty)$ satisfies property \textit{(Q)} by Lemma \ref{lin_P_pass}. We choose a large integer $J$ such that
    $$d_{GH}((J^{-1}C(Z'),z',G'),(C(Z'),z',G'_\infty))\le\eta/4.$$
    Hence for all $i$ large, we have
    $$d_{GH}((J^{-1}l_iN_i,q_i,\Gamma_i),(C(Z'),z',G'_\infty))\le\eta/3.$$
    This implies that $l_i/J\in L_i$ for all $i$ large, a contradiction to our choice of $l_i$ with $\inf L_i\le l_i \le \inf L_i+1/i$..
    
    \textbf{Claim 4:} $(X,K')$ is equivalent to $(X,K_1)$. Suppose not, then we consider the sequence $l_i/2$:
    $$(l_i/2\cdot N_i,q_i,\Gamma_i)\overset{GH}\longrightarrow(1/2 \cdot C(Z'),z',G').$$
    Claim 3 tells us that $(C(Z'),z',G')$ satisfies property \textit{(Q)}; it follows that after a change of scale $(1/2\cdot C(Z'),z',G')$ also satisfies property \textit{(Q)}. This means that $l_i/2\in L_i$ for $i$ large, a contradiction to the choice of $L_i$.
    
    We sum up the results from Claims 1 to 3: for the convergence
    $$(l_iN_i,q_i,\Gamma_i)\overset{GH}\longrightarrow(C(Z'),z',G'),$$
    the limit space $(C(Z'),z',G')$ satisfies property \textit{(Q)}; moreover, $(X,K')$, the $X$-mark of $(Z',\mathrm{Iso}_{z'}G')$, is equivalent to $(X,K_1)$. Together with Lemma \ref{gap_isotropy}, these lead to the ultimate contradiction: Because $l_i\in L_i$, there is some space $(C(Y),y,H)\in\Omega(\widetilde{M},\Gamma)$ satisfying the conditions (1)(2) in the proof of Claim 1, and
    $$d_{GH}((C(Z'),z',G'),(C(Y),y,H))\le\eta/2,$$
    which contradicts Lemma \ref{gap_isotropy}.
\end{proof}

With Lemma \ref{omega_isotropy}, we are ready to prove Theorem \ref{main_omega'}(1,2).

\begin{proof}[Proof of Theorem \ref{main_omega'}(1,2)]
	We show that any space $(C(Z),z,G)\in \Omega(\widetilde{M},\Gamma)$ satisfies property \textit{(Q)}, then by Lemma \ref{omega_isotropy}, the $X$-mark of $(Z,\mathrm{Iso}_z G)$ would be independent of $(C(Z),z,G)$.
	
	Suppose that property \textit{(Q)} fails for some $(C(Z),z,G)\in \Omega(\widetilde{M},\Gamma)$. We consider its equivariant tangent cone at $z$ and at infinity respectively:
	$$(jC(Z),z,G)\overset{GH}\longrightarrow(C(Z),z,G_z),$$
	$$(j^{-1}C(Z),z,G)\overset{GH}\longrightarrow(C(Z),z,G_\infty).$$
	By Lemma \ref{lin_P_pass}, the above two limit spaces satisfy property \textit{(P)}, thus property \textit{(Q)}. We also know that the $X$-marks of $(Z,\mathrm{Iso}_z G)$ and $(Z,\mathrm{Iso}_z G_\infty)$ are not equivalent since property \textit{(Q)} fails on $(C(Z),z,G)$. Together with the fact that $\mathrm{Iso}_z G_z=\mathrm{Iso}_z G$, we now have two spaces in $\Omega(\widetilde{M},\Gamma)$: $(C(Z),z,G_z)$ and $(C(Z),z,G_\infty)$ with property \textit{(Q)} but the $X$-marks of $(Z,\mathrm{Iso}_z G_z)$ and $(Z,\mathrm{Iso}_z G_\infty)$ being not equivalent. This is a contradiction to Lemma \ref{omega_isotropy}.
\end{proof}

\begin{rem}\label{remains}
	We have shown that any space $(C(Z),z,G)\in \Omega(\widetilde{M},\Gamma)$ satisfies property \textit{(Q)}. Together with Lemmas \ref{lin_Q_prop} and \ref{lin_P_decom}, it is clear that the orbit $G\cdot z$ is an $(\mathbb{R}^l\times \mathbb{Z}^m)$-translation orbit. To prove Theorem \ref{main_omega'}(3), it remains to show that $m=0$ and $l$ are the same among all spaces in $\Omega(\widetilde{M},\Gamma)$. Note that $\Omega(\widetilde{M},\Gamma)$ always contains spaces with $\mathbb{R}^l$-translation orbit (without $\mathbb{Z}^m$-factor): this can be done by passing to the tangent cone of any $(C(Z),z,G)\in\Omega(\widetilde{M},\Gamma)$ at $z$. With these observations, it remains to rule out the case that $(\widetilde{M},\Gamma)$ have two spaces $(C(Z_j),z_j,G_j))$ $(j=1,2)$ satisfying\\
    (1) the orbit $G_1\cdot z_1$ is an $l$-dimensional Euclidean subspace,\\
    (2) the orbit $G_2\cdot z_2$ contains an $(l+1)$-dimensional Euclidean subspace, or $G_2\cdot z_2$ contains an $l$-dimensional Euclidean subspace $E$ and an extra orbit point $q$ with $d(E,q)>0$.\\
    Scaling $(C(Z_2),z_2,G_2)$ down by a constant if necessary, we can replace (2) by:\\
    (2') the orbit $G_2\cdot z_2$ contains an $l$-dimensional Euclidean subspace $E$ and an extra orbit point $q$ with $d(E,q)\in (0,1]$.
\end{rem}

\begin{lem}\label{gap_trans}
	Given $X\in\mathcal{M}_{cs}(n,0)$, there exists $\eta(X)>0$ such that the following holds.
	
	For any two metric cones $(C(Z_j),z_j,G_j))$ $(j=1,2)$ satisfying\\
	(1) $d_{GH}(Z_j,X)\le\epsilon_X$,\\
	(2) property \textit{(Q)} holds,\\
	(3) the orbit $G_1\cdot z_1$ is an $l$-dimensional Euclidean subspace,\\
	(4) the orbit $G_2\cdot z_2$ contains an $l$-dimensional Euclidean subspace $E$ and an extra orbit point $q$ with $d(E,q)\in (0,1]$.\\
	Then
	$$d_{GH}((C(Z_1),z_1,G_1),(C(Z_2),z_2,G_2))\ge\eta.$$
\end{lem}

\begin{proof}
	Suppose that the contrary holds, then we have two sequences of metric cones $\{(C(Z_{ij}),z_{ij},G_{ij})\}_i$ $(j=1,2)$ such that\\
	(1) $d_{GH}(Z_{ij},X)\le\epsilon_X$,\\
	(2) property \textit{(Q)} holds for all $i$,\\
	(3) the orbit $G_{i1}\cdot z_{i1}$ is an $l$-dimensional Euclidean subspace,\\
	(4) the orbit $G_{i2}\cdot z_{i2}$ contains an $l$-dimensional Euclidean subspace $E_i$ and an extra point $q_i$ with $d(E_i,q_i)\in (0,1]$,\\
	(5) $d_{GH}((C(Z_{i1}),z_{i1},G_{i1}),(C(Z_{i2}),z_{i2},G_{i2}))\to 0$ as $i\to\infty$.
	
	Passing to a subsequence if necessary, these two sequences converge to the same limit:
	$$(C(Z_{i1}),z_{i1},G_{i1})\overset{GH}\longrightarrow (C(Z),z,G),$$
	$$(C(Z_{i2}),z_{i2},G_{i2})\overset{GH}\longrightarrow (C(Z),z,G).$$
	Since the orbit $G_{ii}\cdot z_{i1}$ converges to the orbit $G\cdot z$ in the limit space and each $G_{i1}\cdot z_{i1}$ is an $l$-dimensional Euclidean subspace, the limit orbit $G\cdot z$ must be an $l$-dimensional Euclidean subspace as well. On the other hand, the orbit $G_{i2}\cdot z_{i2}$ has an extra $\mathbb{Z}$-orbit with generator having displacement $\le 1$. Passing this property to the limit, we see a clear contradiction.
\end{proof}

With Lemma \ref{gap_trans}, we use the critical rescaling argument one more time to finish the proof of Theorem \ref{main_omega'}.

\begin{proof}[Proof of Theorem \ref{main_omega'}(3)]
	As pointed out in Remark \ref{remains}, it suffices to rule out the case that $\Omega(\widetilde{M},\Gamma)$ have two metric cones: 
	\begin{center}
		$(C(Z_1),z_1,G_1)$ and $(C(Z_2),z_2,G_2)$
	\end{center} 
    with the conditions below:\\
	(1) the orbit $G_1\cdot z_1$ is an $l$-dimensional Euclidean subspace,\\
	(2) the orbit $G_2\cdot z_2$ contains an $l$-dimensional Euclidean subspace $E$ and an extra orbit point $q$ with $d(E,q)\in (0,1]$.\\
	We also choose $(C(Z_1),z_1,G_1)$ so that its orbit $G_1\cdot z_1$ has the smallest dimension among all spaces in $\Omega(\widetilde{M},\Gamma)$.
	
	Let $r_i\to\infty$ and $s_i\to\infty$ be two sequences such that
	$$(r^{-1}_i\widetilde{M},\tilde{x},\Gamma)\overset{GH}\longrightarrow(C(Z_1),z_1,G_1),$$
	$$(s^{-1}_i\widetilde{M},\tilde{x},\Gamma)\overset{GH}\longrightarrow(C(Z_2),z_2,G_2).$$
	Passing to a subsequence, we assume that $t_i:=(s^{-1}_i)/(r^{-1}_i)\to\infty$. We put $$(N_i,q_i,\Gamma_i)=(r^{-1}_i\widetilde{M},\tilde{x},\Gamma),$$
	then
	$$(N_i,q_i,\Gamma_i)\overset{GH}\longrightarrow(C(Z_1),z_1,G_1),$$
	$$(t_iN_i,q_i,\Gamma_i)\overset{GH}\longrightarrow(C(Z_2),z_2,G_2).$$
	
	Let $\eta>0$ be the constant in Lemma \ref{gap_trans}. For each $i$, we define a set of scales
	\begin{align*}
		L_i:=\{\  1\le l\le t_i\ |\ & d_{GH}((l{N}_i,q_i,\Gamma_i),({W},{w},H))\le \eta/3 \\
		& \text{ for some space $(C(Y),y,H)\in \Omega(\widetilde{M},\Gamma)$ such that}\\
		& \text{ the orbit $H\cdot y$ contains an $l$-dimensional Euclidean }\\
		& \text{ subspace $E$ and an extra point with $d(E,q)\in (0,1]$. }\}
	\end{align*}
	We know that $t_i\in L_i$ for $i$ large. We choose $l_i\in L_i$ of $\inf L_i\le l_i \le \inf L_i+1/i$.
	
	\textbf{Claim 1:} $l_i\to\infty$. If $l_i\to B<\infty$, then
	$$(l_iN_i,q_i,\Gamma_i)\overset{GH}\longrightarrow(B\cdot C(Z_1),z_1,G_1).$$
	The space $(B\cdot C(Z_1),z_1,G_1)$ satisfies property \textit{(Q)}, and the orbit $G_1\cdot z_1$ is an $l$-dimensional Euclidean subspace. Since $l_i\in L_i$, by the definition of $L_i$ and the convergence, we have
	$$d_{GH}((B\cdot C(Z_1),z_1,G_1),(C(Y),y,H))\le\eta/2$$
	for some $(C(Y),y,H)\in \Omega(\widetilde{M},\Gamma)$ with the prescribed conditions. This is a contradiction to Lemma \ref{gap_trans}.
	
	Next we consider convergence
	$$(l_iN_i,q_i,\Gamma_i)\overset{GH}\longrightarrow(C(Z'),z',G').$$
	
	\textbf{Claim 2:} The orbit $G'\cdot z'$ is an $l$-dimensional Euclidean subspace in $C(Z')$. Indeed, because $(C(Z'),z',G')$ satisfies property \textit{(Q)}, we know that the orbit $G'\cdot z'$ is a translation $(\mathbb{R}^{l'}\times \mathbb{Z}^{m'})$-orbit. If $l'<l$, then we pass to the equivariant tangent cone of $(C(Z'),z',G')$ at $z'$. In such a limit space, the orbit is an $l'$-dimensional Euclidean subspace with $l'<l$, which contradicts with our choice of $(C(Z_1),z_1,G_1)$. If $l'>l$, or $l'=l$ but $m'\not= 0$, then the orbit $G'\cdot z'$ contains an $l$-dimensional Euclidean subspace $E'$ and an extra orbit point $q'$. Let $d>0$ be the distance between $E'$ and $q'$. If $d\le 1$, then $l_i/2\in L_i$. If $d>1$, then $l_i/(2d)\in L_i$. In either case, we see a contradiction to our choice of $l_i$. Hence Claim 2 holds.
	
	We derive the desired contradiction: $l_i\in L_i$ so by the definition of $L_i$ and the convergence, $$d_{GH}((C(Z'),z',G'),(C(Y),y,H))\le\eta/2$$
	for some space $(C(Y),y,H)\in \Omega(\widetilde{M},\Gamma)$ with the prescribed conditions, a contradiction to Lemma \ref{gap_trans}. 
\end{proof}

\section{Virtually abelian structure}

The goal of this section is the theorem below, which implies that $\pi_1(M)$ is virtually abelian if $\widetilde{M}$ is $(C(X),\epsilon_X)$-stable at infinity or $k$-Euclidean at infinity.

\begin{thm}\label{main_abel}
	Let $(M,x)$ be an open $n$-manifold with $\mathrm{Ric}\ge 0$ and a finitely generated nilpotent fundamental group $\Gamma=\pi_1(M,x)$. Suppose that\\
	(1) any tangent cone of $\widetilde{M}$ at infinity $(Y,y)$ is a metric cone with vertex $y$,\\
	(2) for any equivariant tangent cone of $(\widetilde{M},\Gamma)$ at infinity
	$$(C(Z),z,G)=(\mathbb{R}^k\times C(Z'),(0,z'),G),$$
	where $\mathrm{diam}(Z')<\pi$,
	the projection of $G$-action to the maximal Euclidean factor $(\mathbb{R}^k,0,p_Z(G))$ satisfies property \textit{(P)}.\\
	Then the $Z(\Gamma)$, the center of $\Gamma$, has finite index in $\Gamma$.
\end{thm}

\begin{cor}\label{cor_v_abel}
	Let $M$ be an open $n$-manifold of $\mathrm{Ric}\ge 0$. Suppose that one of the following conditions is true:\\
	(1) $\widetilde{M}$ is $(C(X),\epsilon_X)$-stable at infinity for some $X\in \mathcal{M}_{cs}(n,0)$ and the corresponding constant $\epsilon_X>0$ as in Proposition \ref{to_model}; or\\
	(2) $\widetilde{M}$ is $k$-Euclidean at infinity.\\
	Then $\pi_1(M)$ contains a normal abelian subgroup of finite index.
\end{cor}

\begin{proof}
	Under the condition (1) or (2), we know that $\Gamma=\pi_1(M,x)$ is finitely generated. By the work of Milnor and Gromov \cite{Mi,Gro2}, $\Gamma$ contains a normal nilpotent subgroup $N$ of finite index. For this subgroup $N$, we apply Theorem \ref{main_omega'} and Lemma \ref{lin_Q_prop} for condition (1), or \cite{Pan2} for condition (2). It follows that under either condition, for any equivariant tangent cone of $(\widetilde{M},N)$ at infinity $(C(Z),z,G)$, assumption (2) in Theorem \ref{main_abel} is fulfilled. Thus by Theorem \ref{main_abel}, $Z(N)$ has finite index in $N$. Consequently, $Z(N)$ is a normal abelian subgroup of $\Gamma$ with finite index. 
\end{proof}

We start with some commutator computations.

\begin{lem}\label{id_comm}
	Let $\alpha,\beta$ be two elements in a group $G$. Then for any integer $l\ge 2$,\\
	(1) $[\alpha^l,\beta]=\left(\prod_{j=1}^{l-1}[\alpha,[\alpha^{l-j},\beta]]\right)\cdot [\alpha,\beta]^l$,\\
	(2) $[\alpha,\beta^l]=[\alpha,\beta]^l\cdot \left(\prod_{j=1}^{l-1}[[\beta^{j},\alpha],\beta]\right)$.
\end{lem}

\begin{proof}
	It is straight-forward to verify that
	$$[\alpha\beta,\gamma]=[\alpha,[\beta,\gamma]]\cdot[\beta,\gamma]\cdot[\alpha,\gamma].$$
	In fact,
	\begin{align*}
		\text{RHS}=&\alpha [\beta,\gamma] \alpha^{-1} [\beta,\gamma]^{-1}\cdot [\beta,\gamma] \cdot [\alpha,\gamma]\\
		=&\alpha [\beta,\gamma] \alpha^{-1}\cdot \alpha\gamma\alpha^{-1}\gamma^{-1}\\
		=&\alpha\beta\gamma\beta^{-1}\gamma^{-1}\cdot \gamma\alpha^{-1}\gamma^{-1}\\
		=&(\alpha\beta)\gamma (\alpha\beta)^{-1}\gamma^{-1}=\text{LHS}.
	\end{align*}
	
	We use this identity repeatedly to prove (1):
	\begin{align*}
		[\alpha^l,\beta]=&[\alpha\cdot \alpha^{l-1},\beta]\\
		=&[\alpha,[\alpha^{l-1},\beta]]\cdot[\alpha^{l-1},\beta]\cdot [\alpha,\beta]\\
		=&[\alpha,[\alpha^{l-1},\beta]]\cdot[\alpha,[\alpha^{l-2},\beta]]\cdot[\alpha^{l-2},\beta]\cdot [\alpha,\beta]^2\\
		=&\left(\prod_{j=1}^{l-1}[\alpha,[\alpha^{l-j},\beta]]\right)\cdot [\alpha,\beta]^l.
	\end{align*}
	
	(2) follows from (1):
	\begin{align*}
		[\alpha,\beta^l]=&[\beta^{l},\alpha]^{-1}\\
		=&\left(\left(\prod_{j=1}^{l-1}[\beta,[\beta^{l-j},\alpha]]\right)\cdot [\beta,\alpha]^l\right)^{-1}\\
		=& [\beta,\alpha]^{-l}\cdot \left(\prod_{j=1}^{l-1}[\beta,[\beta^{l-j},\alpha]]\right)^{-1}\\
		=& [\alpha,\beta]^l \cdot \prod_{j=1}^{l-1}[\beta,[\beta^{j},\alpha]]^{-1}\\
		=& [\alpha,\beta]^l \cdot \prod_{j=1}^{l-1}[[\beta^{j},\alpha],\beta].
	\end{align*}
\end{proof}

For the remaining of this section, we always assume that $\Gamma=\pi_1(M,x)$ is nilpotent with the assumptions in Theorem \ref{main_abel}.

\begin{lem}\label{id_nil}
	Let $\alpha\in \Gamma$ and $\beta\in C_{k-1}(\Gamma)$. Then for any integer $k$
	$$[\alpha^l,\beta^l]=[\alpha,\beta]^{(l^2)}\cdot h$$
	for some element $h\in C_{k+1}(\Gamma)$.
\end{lem}

\begin{proof}
	First note that for any $\alpha\in \Gamma$ and $\beta\in C_{k-1}(\Gamma)$, the terms $$\prod_{j=1}^{l-1}[\alpha,[\alpha^{l-j},\beta]], \quad \prod_{j=1}^{l-1}[[\beta^{j},\alpha],\beta]$$
	in Lemma \ref{id_comm} are elements in $C_{k+1}(\Gamma)$. Also note that $C_{k+1}(\Gamma)$ is normal in $\Gamma$, thus for any $\gamma\in \Gamma$ and any $h\in C_{k+1}(\Gamma)$, we have
	$h\cdot \gamma=\gamma\cdot h'$ for some $h'\in C_{k+1}(\Gamma)$.
	
	With these observations and Lemma \ref{id_comm}, we calculate $[\alpha^l,\beta^l]$ as follows:
	\begin{align*}
		[\alpha^l,\beta^l]&=[\alpha^l,\beta]^l \cdot \left(\prod_{j=1}^{l-1}[[\beta^{j},\alpha^l],\beta]\right)\\
		&=\left(\left(\prod_{j=1}^{l-1}[\alpha,[\alpha^{l-j},\beta]]\right)\cdot [\alpha,\beta]^l \right)^l\cdot h\\
		&=(h'\cdot [\alpha,\beta]^l)^l\cdot h\\
		&=[\alpha,\beta]^{(l^2)}\cdot h'',
	\end{align*}
	where $h,h',h''$ are elements in $C_{k+1}(\Gamma)$.
\end{proof}

For simplicity, we write $|\gamma|=d(\gamma\tilde{x},\tilde{x})$ for any element $\gamma\in \Gamma$. We show that if $|\gamma|$ is sufficiently large, then $\gamma$ behaves almost as a translation at $\tilde{x}$.

\begin{lem}\label{almost_trans}
	There exists $R_0>0$ such that for any $\gamma\in\Gamma$, if $|\gamma|\ge R_0$, then
	$$|\gamma^2|\ge 1.9\cdot|\gamma|.$$
\end{lem}

\begin{proof}
	We argue by contradiction. Suppose that there are a sequence of element $\gamma_i\in \Gamma$ such that $r_i=|\gamma_i|\to\infty$, but
	$$|\gamma_i^2|<1.9\cdot |\gamma_i|.$$
	We consider the tangent cone of $\widetilde{M}$ at infinity coming from the sequence $r_i^{-1}\to 0$: passing to a subsequence, we have
	$$(r_i^{-1}\widetilde{M},\tilde{x},\Gamma,\gamma_i)\overset{GH}\longrightarrow(C(Z),z,G,\gamma_\infty).$$
	$(C(Z),z)$ splits isometrically as $(\mathbb{R}^k\times C(Z'),(0,z'))$ with $\mathrm{diam}(Z')<\pi$. Since $(\mathbb{R}^k,0,p_Z(G))$ satisfies property \textit{(P)}, according to Lemma \ref{lin_P_decom}, we know that
	$$p_Z(G)=\mathrm{Iso}_0 p_Z(G)\times T,$$
    and the subgroup $\{e\}\times T$ acts as translations in the $\mathbb{R}^k$-factor. Due to our choice of $r_i$, $\gamma_\infty$ has displacement $1$ at $z$. Therefore, we can write $p_Z(\gamma_\infty)$ as
    $$p_Z(\gamma_\infty)=(A,v)\in \mathrm{Iso}_0 p_Z(G)\times T,$$
    where $v$ is a vector of length $1$. Note that $p_Z(\gamma_\infty^2)=(A^2,2v)$, then we see that $$d(\gamma_\infty^2\cdot z,z)=2.$$
	On the other hand, because $|\gamma_i^2|<1.9\cdot |\gamma_i|$, $\gamma_\infty^2$ should satisfy
	$$d(\gamma_\infty^2\cdot z,z)\le 1.9,$$
	a contradiction.
\end{proof}

\begin{cor}\label{almost_trans_cor}
	Let $R_0$ be the constant in Lemma \ref{almost_trans}. Suppose that $\gamma\in \Gamma$ with $|\gamma|\ge R_0$, then for any integer $m$, we have
	$$|\gamma^{(2^m)}|\ge 1.9^m \cdot |\gamma|.$$
\end{cor}

\begin{lem}\label{finite_comm}
	For $k\ge 1$, if $C_{k+1}(\Gamma)$ is finite, then $C_{k}(\Gamma)$ is also finite.
\end{lem}

\begin{proof}
	We claim that every element in $C_{k}(\Gamma)$ of form $[\alpha,\beta]$ has finite order, where $\alpha\in \Gamma$ and $\beta\in C_{k-1}(\Gamma)$. If this claim holds, then $C_{k}(\Gamma)$ is generated by elements of finite order. Recall that for a finitely generated nilpotent group, all elements of finite order form a finite subgroup of $\Gamma$, known as the torsion subgroup $\mathrm{Tor}(\Gamma)$. Hence $C_{k}(\Gamma)$, as a subgroup of $\mathrm{Tor}(\Gamma)$, must be finite as well. It remains to verify the claim.
	
	We argue by contradiction. Suppose that there are $\alpha\in \Gamma$ and $\beta \in C_{k-1}(\Gamma)$ such that $[\alpha,\beta]$ has infinite order. Because $\Gamma$ acts freely and discretely on $\widetilde{M}$, we can choose a large integer $l$ such that $$|[\alpha,\beta]^{(l^2)}|\ge R_0,$$
	where $R_0$ is the constant in Lemma \ref{almost_trans}. By Lemma \ref{id_nil},
	$$[\alpha^l,\beta^l]=[\alpha,\beta]^{(l^2)} \cdot h$$
	for some $h\in C_{k+1}(\Gamma)$. We put $g_1=\alpha^l$, $g_2=\beta^l$ and $\gamma=[\alpha,\beta]^{(l^2)}$, then we have
	$$[g_1,g_2]=\gamma\cdot h$$
	for some element $h\in C_{k+1}(\Gamma)$.
	To derive a contradiction, we continue to raise the power. For any integer $p$, we have
	\begin{align*}
		[g_1^p,g_2^p]&=[g_1,g_2]^{(p^2)}\cdot h'\\
		&=(\gamma\cdot h)^{(p^2)}\cdot h'\\
		&=\gamma^{(p^2)}\cdot h'',
	\end{align*}
	where $h'$ and $h''$ are certain elements in $C_{k+1}(\Gamma)$. We estimate the length of each sides when $p=2^m$. The left hand side has length
	$$|[g_1^{(2^m)},g_2^{(2^m)}]|\le 2\left(|g_1^{(2^m)}|+|g_2^{(2^m)}|\right)
	\le 2^{m+1}\cdot\left(|g_1|+|g_2|\right).$$
	Due to Lemma \ref{almost_trans_cor}, the right hand side has length
	$$|\gamma^{((2^m)^2)}\cdot h''|\ge |\gamma^{(2^{2m})}|-|h''|\ge 1.9^{2m}|\gamma|-|h''|\ge 1.9^{2m} |\gamma|-D,$$
	where $D=\max_{h\in C_{k+1}(\Gamma)} |h|<\infty$ because $C_{k+1}(\Gamma)$ is finite.
	
	As a result, it follows that
	$$1.9^{2m} |\gamma|-D\le 2^{m+1}\cdot\left(|g_1|+|g_2|\right)$$
	for any integer $m>0$. Since $|\gamma|,D,|g_1|$ and $|g_2|$ all have fixed values, we see a clear contradiction when $m$ is sufficiently large. This verifies the claim and completes the proof.
\end{proof}

\begin{proof}[Proof of Theorem \ref{main_abel}]
	Let $k$ be the nilpotency length of $\Gamma$. We apply Lemma \ref{finite_comm} inductively starting from $C_{k+1}(\Gamma)=\{e\}$; we conclude that $C_1(\Gamma)=[\Gamma,\Gamma]$ is finite. For any finitely generated group, $[\Gamma,\Gamma]$ being finite implies that the center of $\Gamma$ has finite index in $\Gamma$, which is a standard result in group theory. In particular, $\Gamma$ is virtually abelian. For completeness, we include the proof as below.
	
	Let $g_1,...,g_m$ be a set of generators of $\Gamma$. Let $Z(g_j)$ be the subgroup that consists of elements in $\Gamma$ commuting with $g_j$. It is clear that $Z(G)=\cap_{j=1}^m Z(g_j)$. Thus it suffices to show that each $Z(g_j)$ has finite index in $\Gamma$, or equivalently, there are only finitely many elements in $\Gamma$ conjugating to $g_j$. Indeed, for any $\gamma\in \Gamma$, the conjugation of $g_j$ under $\gamma$ is
	$$\gamma g_j \gamma^{-1}=[\gamma,g_j]\cdot g_j,$$
	which only has finitely options since $[\Gamma,\Gamma]$ is finite. This shows that $[\Gamma:Z(g_j)]<\infty$ for each $j$ and $[\Gamma:Z(\Gamma)]<\infty$ follows.
\end{proof}

\section{Euclidean volume growth and bounded index}

In this section, we show that in Corollary \ref{cor_v_abel} if in addition $\widetilde{M}$ has Euclidean volume growth, then the index can be uniformly bounded in terms of $n$ and volume growth constant. From the proof in Section 4, we see that the key is showing the finiteness of $[\Gamma,\Gamma]$ for a finitely generated nilpotent fundamental group. We prove that when $\widetilde{M}$ has Euclidean volume growth, the finiteness of $[\Gamma,\Gamma]$ turns into a uniform bound on its number.

\begin{thm}\label{main_index}
	Given $n\in\mathbb{N}$ and $L\in(0,1]$, there exists a constant $C(n,L)$ such that the following holds. 
	
	Let $M$ be an open $n$-manifold of $\mathrm{Ric}\ge 0$. Suppose that\\
	(1) $\widetilde{M}$ has Euclidean volume growth of constant at least $L$,\\
	(2) $\Gamma=\pi_1(M,x)$ is finitely generated and nilpotent with nilpotency length $\le n$,\\
	(3) $\#[\Gamma,\Gamma]$ is finite.\\
	Then $\#[\Gamma,\Gamma]\le C(n,L)$.
\end{thm}

\begin{cor}\label{cor_C_abel}
	Given $n\in\mathbb{N}$ and $L\in(0,1]$, there exists a constant $C(n,L)$ such that the following holds.
	
	Let $M$ be an open manifold of $\mathrm{Ric}\ge 0$. Suppose that $\widetilde{M}$ has Euclidean volume growth of constant at least $L$ and one of the following conditions is true:\\
	(1) $\widetilde{M}$ is $(C(X),\epsilon_X)$-stable at infinity for some $X\in \mathcal{M}_{cs}(n,0)$ and the corresponding constant $\epsilon_X>0$ as in Proposition \ref{to_model}; or\\
	(2) $\widetilde{M}$ is $k$-Euclidean at infinity.\\
	Then $\pi_1(M)$ contains a normal abelian subgroup of index at most $C(n,L)$.
\end{cor}

\begin{proof}
	By \cite{KW}, $\Gamma=\pi_1(M,x)$ contains a normal nilpotent subgroup $N$ of index at most $C_1(n)$ and of nilpotency length at most $n$. It suffices to bound $[N:Z(N)]$ in terms of $n$ and $L$. 
	
	By the proof of Theorem \ref{main_abel}, it is easy to see that $[N:Z(N)]$ can be bounded by a constant only involving $\#[N,N]$ and the number of generators of $N$. Since we can uniformly bound the number of generators by some constant $C_2(n)$ \cite{KW}, the result now follows from Theorem \ref{main_index}.
\end{proof}

To prove Theorem \ref{main_index}, we again investigate equivariant tangent cones of $(\widetilde{M},\Gamma)$ at infinity. We need a quantitative version of the no small subgroup property for isometry group of any non-collapsing Ricci limit space \cite{PR}.

\begin{prop}\label{nss}\cite{PR}
	Given $n,v>0$, there exists a positive constant $\delta(n,v)$ such that for any Ricci limit space $(X,x)\in\mathcal{M}(n,-1,v)$
	and any nontrivial subgroup $H$ of $\mathrm{Isom}(X)$, we have $$D_{1,x}(H)\ge \delta(n,v),$$
	where
	$$D_{1,x}(H)=\sup_{y\in B_1(x),h\in H} d(hy,y).$$
\end{prop}

\begin{lem}\label{lem_index}
	Let $M$ be an open $n$-manifold with the assumptions in Theorem \ref{main_index}. For any sequence $r_i\to\infty$, we consider $$(r_i^{-1}\widetilde{M},\tilde{x},[\Gamma,\Gamma],\Gamma)\overset{GH}\longrightarrow (C(Z),z,H,G).$$
	Then\\
	(1) $\#H=\#[\Gamma,\Gamma]$;\\
	(2) $H\subseteq [K,K]$, where $K=\mathrm{Iso}_zG$.
\end{lem} 

\begin{proof}
	(1) Since $[\Gamma,\Gamma]$ is a finite group, its limit $H$ must be a finite group fixing $z$.
	Let $m=\#[\Gamma,\Gamma]$. By Proposition \ref{nss},
	$$D_{1,\tilde{x}}(g)\ge \delta(n,\omega_n L)/m$$
	for all $g\in [\Gamma,\Gamma]$ on any $r_i^{-1}\widetilde{M}$, where $\omega_n$ is the volume of a unit ball in $\mathbb{R}^n$. This implies that $H$ also has order $m$.
	
	(2) Because $[\Gamma,\Gamma]$ is a finite group, we can write each $g_j\in [\Gamma,\Gamma]$ as
	$$g_j=\prod_{t=1}^{l_j}[\alpha_{tj},\beta_{tj}]$$
	for some $\alpha_{tj},\beta_{tj}\in \Gamma$. Passing to a subsequence if necessary, for each $\alpha_{tj}$ and each $\beta_{tj}$, we have convergence
	$$(r_i^{-1}\widetilde{M},\tilde{x},\alpha_{tj},\beta_{tj})\overset{GH}\longrightarrow (C(Z),z,\alpha^\infty_{tj},\beta^\infty_{tj}),\quad t=1,...,l_j,\quad j=1,...,m,$$
	where $\alpha^\infty_{tj}$ and $\beta^\infty_{tj}$ are elements in $K$. It is clear that $H$ consists of elements
	$$g^\infty_j=\prod_{t=1}^{l_j}[\alpha^\infty_{tj},\beta^\infty_{tj}]\in [K,K],\quad j=1,...,m.$$
\end{proof}

From Lemma \ref{lem_index}, we see that if we can find a uniform bound $C(n,L)$ so that $\#[K,K]\le C(n,L)$, then we would have the desired bound for $\#[\Gamma,\Gamma]$:
$$\#[\Gamma,\Gamma]=\# H \le \# [K,K] \le C(n,L).$$
Also note that $K$ is a closed nilpotent Lie subgroup of $\mathrm{Isom}(Z)$ with nilpotency length at most $n$, where $Z\in\mathcal{M}_{cs}(n,0,\omega_nL)$. Hence Theorem \ref{main_index} follows directly from the lemma below.

\begin{lem}\label{comm_bound}
	Given $n,v>0$, there exists constants $C_1(n,v)$ and $C_2(n,v)$ such that the following holds.
	
	Let $Z\in \mathcal{M}_{cs}(n,0,v)$ be a space and let $K$ be a closed nilpotent Lie subgroup of $\mathrm{Isom}(Z)$ with nilpotency length at most $n$. Then\\
	(1) $[K:Z(K)]\le C_1(n,v)$,\\
	(2) $\#[K,K]\le C_2(n,v)$.
\end{lem}

To prove Lemma \ref{comm_bound}, we recall a classical result in group theory: Schur's Lemma.

\begin{lem}\label{Schur_lemma}
	(Schur) Let $G$ be a group. If $[G:Z(G)]=k$ is finite, then
	$[G,G]$ is a finite group of at most $C(k)=k^{(2k^3)}$ many elements.
\end{lem}

We also need a structure result for compact nilpotent Lie groups.

\begin{lem}\label{G0_central}
	Let $G$ be a compact nilpotent Lie group, then $G_0$, the identity component of $G$, is central in $G$.
\end{lem}

\begin{proof}
	The proof is standard in group theory. We include the proof for the convenience of readers.\footnote{The author would like to thank user YCor on Mathoverflow site for providing this proof. http://mathoverflow.net/questions/255976/compact-non-connected-nilpotent-lie-subgroup-of-on.}
	
	Recall that $G_0$, as a connected compact nilpotent Lie group, must be a torus. We put $F=G/G_0$, which is a finite group. For every $\gamma\in F$, we have an automorphism of $G_0$ derived by conjugation of $\gamma$, and thus an automorphism of the Lie algebra $\mathrm{Lie}(G_0)=\mathbb{R}^l$. This provides a representation of $F$ on $\mathbb{R}^l$
	$$\phi: F \to \mathrm{GL}(l,\mathbb{R}).$$
	Let $V$ be the subspace of $\mathbb{R}^l$ where $(F,\phi)$ has trivial representation:
	$$V:=\{v\in\mathbb{R}^l\ |\ \phi(\gamma)v=v \text{ for all } \gamma\in F\}.$$
	It is not difficult to show that $V\not=\{e\}$ because $G$ is nilpotent. In fact, notice that the set
	$$\{[h,t]\ |\ h\in G, t\in G_0\}$$
	is connected and is a subset of $G_0$, thus $[G,G_0]$ is connected. By the same reason, each $[G,[...,[G,G_0]$ is connected. Together with the fact that $G$ is nilpotent, we can find a circle subgroup lying in the center of $G$. This verifies that $V\not=0$.
	
	Since the representation $(F,\phi)$ is completely reducible, $\mathbb{R}^l$ splits as $V\oplus W$, where $W$ is a complement of $V$ and is invariant under $\phi(F)$. The invariance of $W$ shows that $$L:=\{-w+\phi(\gamma)w\ |\ w\in W, \gamma\in F\}$$ is a $\phi(F)$-invariant subspace of $W$. If $W\not=0$, then $L\not=0$. Consider the connected subgroup $T=\overline{\mathrm{exp}(L)}$ in $G_0$. By the nilpotency of $G$ and the same argument that we showed $V\not= 0$, we verify that $L\cap V\not=0$. This contradicts with $W\cap V=0$. Therefore, $W=0$, $V=\mathbb{R}^l$ and this finishes the proof.
\end{proof}

\begin{proof}[Proof of Lemma \ref{comm_bound}]
	(1) We argue by contradiction. Let $Z_i$ be a sequence of spaces in $\mathcal{M}_{cs}(n,0,v)$, and for each $i$ let $K_i$ be a closed nilpotent Lie subgroup of $\mathrm{Isom}(Z)$ with nilpotency length at most $n$. By Lemma \ref{G0_central}, we know that $[K_i:Z(K_i)]$ is finite for each $i$. 
	
	Suppose that $[K_i:Z(K_i)]\to\infty$ as $i\to\infty$. Passing to a subsequence if necessary, we obtain equivariant convergence
	$$(Z_i,K_i)\overset{GH}\longrightarrow (Z_\infty,K_\infty)$$
	for some limit space $Z_\infty\in \mathcal{M}_{cs}(n,0,v)$. 
	
	Since $K_\infty$ is a compact Lie group, we can apply Corollary \ref{good_map_cor} to obtain a sequence of $\epsilon_i$-approximated homomorphisms $\psi_i:K_i\to K_\infty$ for some $\epsilon_i\to 0$. Let $H_i$ be the kernel of $\psi_i$, then
	$$(Z_i,H_i)\overset{GH}\longrightarrow(Z_\infty,\{e\}).$$
	$H_i$ must be the trivial for all $i$ large. Otherwise, by Proposition \ref{nss}, there is $\delta>0$ such that $D(H_i)\ge \delta$ for all $i$, where $D(H_i)$ is the displacement of $H_i$ on $Z_i$; thus $H_i$ can not converge to $\{e\}$. This shows that $\psi_i$ is injective for all $i$ large. With these maps and the fact that $(K_\infty)_0$ is central in $K_\infty$, we see a uniform bound
	$$[K_i:Z(K_i)]\le [K_i:\psi^{-1}((K_\infty)_0)]\le[K_\infty: (K_\infty)_0]<\infty$$
	for all $i$. We obtain the desired contradiction and complete the proof.
	
	(2) follows directly from (1) and Schur's Lemma.
\end{proof}

\Addresses

\end{document}